\documentclass[11pt]{article}
\usepackage[utf8]{inputenc}
\usepackage{amsmath,amssymb,amsthm}
\usepackage[shortlabels]{enumitem}
\usepackage{xcolor, wasysym}
\usepackage{multirow}
\usepackage{comment}
\usepackage[left=1.2in,right=1.2in]{geometry}
\usepackage{graphicx,import}

\usepackage[colorlinks,citecolor=black,linkcolor=black]{hyperref}
\usepackage[nameinlink]{cleveref}

\newcommand\red[1]{{\color{red}#1}}
\newcommand\blue[1]{{\color{blue}#1}}

\newcounter{mcomments}

\newcounter{gcomments}

\newcounter{kcomments}

\usepackage{tcolorbox} 
\newlist{todolist}{itemize}{2}
\setlist[todolist]{label=$\square$}
\usepackage{pifont}
\theoremstyle{plain}
\newtheorem{THM}{Theorem}[section]
\newtheorem{thm}[THM]{Theorem}
\newtheorem{PROP}[THM]{Proposition}
\newtheorem{prop}[THM]{Proposition}
\newtheorem{LEM}[THM]{Lemma}
\newtheorem{lemma}[THM]{Lemma}
\newtheorem{COR}[THM]{Corollary}

\newtheorem{CLAIM}{Claim}

\theoremstyle{definition}
\newtheorem{DEF}[THM]{Definition}
\newtheorem{definition}[THM]{Definition}
\newtheorem{RMK}[THM]{Remark}
\newtheorem{EX}[THM]{Example}
\newtheorem{example}[THM]{Example}
\newtheorem{Question}[THM]{Question}

\newtheorem{remark}[THM]{Remark}
\newtheorem*{assumptions}{Standing Assumptions}

\newtheorem{innercustomthm}{Theorem}
\newenvironment{customthm}[1]
  {\renewcommand\theinnercustomthm{#1}\innercustomthm}
  {\endinnercustomthm}

\DeclareMathOperator{\Aut}{Aut}

\DeclareMathOperator{\Homeo}{Homeo}
\DeclareMathOperator{\Map}{Map}
\DeclareMathOperator{\Stab}{Stab}

\DeclareMathOperator{\supp}{supp}

\newcommand{\Z}{\mathbb{Z}}
\newcommand{\R}{\mathbb{R}}
\newcommand{\N}{\mathbb{N}}
\newcommand{\Q}{\mathbb{Q}}

\newcommand{\cA}{\mathcal{A}}
\newcommand{\cB}{\mathcal{B}}
\newcommand{\cC}{\mathcal{C}}

\newcommand{\cU}{\mathcal{U}}
\newcommand{\cV}{\mathcal{V}}
\newcommand{\cZ}{\mathcal{Z}}

\def\S{{\Sigma}}
\def\O{{\Omega}}

\newcommand{\defeq}{:=}

\newcommand{\homeo}{\mathrel{\cong}} 

\title{Classification of Stable Surfaces with respect to Automatic Continuity}
\author{Mladen Bestvina, George Domat, Kasra Rafi}
\date{\today}

\begin{document}

\maketitle

\begin{abstract}
We provide a complete classification of when the homeomorphism group of a stable surface, $\Sigma$, has the 
automatic continuity property: Any homomorphism from $\Homeo(\S)$ to a separable group is necessarily 
continuous. This result descends to a classification of when the mapping class group of $\Sigma$ 
has the automatic continuity property. Towards this classification, we provide a general framework for proving 
automatic continuity for groups of homeomorphisms. Applying this framework, we also show that the homeomorphism 
group of any stable second countable Stone space has the automatic continuity property. Under the presence of 
stability this answers two questions of Mann. 
\end{abstract}

\section{Introduction}

Given a topological group, a natural question is: How does the algebra
of the group determine the topology of the group? Surprisingly, the
answer to this question can be ``almost entirely." A topological group, $G$, has
the \textbf{automatic continuity property} if every group homomorphism
from $G$ to a separable group $H$ is continuous. Recent examples of
groups with this property are the homeomorphism groups of the Cantor
set, $\R$ \cite{RS2007}, compact manifolds \cite{Rosendal2008,
  Mann2016}, compact manifolds with a Cantor set and a finite set removed \cite{Mann2024}, and some infinite-type surfaces \cite{Vlamis2024}. There is also a classification theorem \cite{Dickmann2023} for pure mapping class groups of infinite-type surfaces (allowing for noncompact boundary components). See \cite{Rosendal2009,RCS2024} for surveys of the question together with some historical context. 

In this paper, we prove a classification theorem for the homeomorphism
and mapping class groups of a large class of infinite-type
surfaces. For the collection of stable surfaces (see Definitions~\ref{stable} and \ref{stable-surface}) this gives a complete answer two questions of Mann (\cite[Question
  4.8]{AIMPL},\cite[Question 2.4]{Mann2024}). All our surfaces $\S$
will be connected, orientable and without boundary.

\begin{customthm}{A} \label{thm:surfaceclassification}
	Let $\S$ be a connected, stable, orientable surface without boundary. The homeomorphism group and mapping class group of $\S$ have automatic continuity if and only if every end of $\S$ is telescoping. That is, every end is one of the following:
	\begin{enumerate}
		\item An isolated puncture,
		\item of Cantor type, or
		\item is not isolated in the space of ends accumulated by genus and is a successor with all predecessors of Cantor type. 
	\end{enumerate}
\end{customthm}

The stability assumption on $\S$  rules out various pathological phenomena. In particular, stability ensures that $\Homeo(\S)$-orbits
  of ends are locally closed and hence each orbit is either locally a
  Cantor set or consists of isolated points
  (\Cref{lem:Cantor-type}), $\S$ has only finitely many equivalence classes of maximal ends (\Cref{prop:stablefinmax}), and allows us to prove an alternate characterization of ``telescoping'' (\Cref{lem:telnbhdannuli}). We make use of each of these in order to obtain a complete classification. In \Cref{ssec:unknownex} we offer an unstable surface for which we do not know whether its mapping
  class group satisfies automatic continuity. 

Here we make ample use of the language and perspective provided by \cite{MR2023}, particularly with respect to the partial order on the space of ends of a surface. The techniques used to prove the positive direction are extensions of those in \cite{RS2007, Rosendal2008, Mann2016, Mann2024}. In particular, the ``telescoping" condition can be thought of as requiring each end to have a neighborhood that exhibits similar behavior to a punctured disk. This condition allows one to decompose any neighborhood of an end into infinitely many homeomorphic ``annuli" such that one can shift the collection of annuli either towards or away from the end. We also use an Eilenberg-Mazur swindle \cite{Bass1963,Mazur1959} in order to deal with the issue of locally writing certain homeomorphisms as commutators. This type of argument does not appear in the previous proofs of automatic continuity.
For the negative direction we build on tools developed in \cite{Domat2022}. All of the discontinuous maps we exhibit factor through $\Map(\S)$ and are constructed in \cite{Domat2022} via actions on Gromov hyperbolic metric spaces. 

\begin{RMK}\label{rmk:failure}
	We note that some of our techniques used to prove the failure
        of automatic continuity do not require the surface to be
        stable. In particular, if $\Sigma$ (not necessarily stable)
        has an end that (1) has a countable $\Homeo(\S)$-orbit, and
        (2) has either a countable predecessor or is isolated in the
        space of ends accumulated by genus, then $\Homeo(\S)$ and
        $\Map(\S)$ will fail to have automatic continuity. This
        follows by directly applying the techniques of
        \Cref{ssec:acfailure}. Similarly, if $\Sigma$ has an end with
        a countable orbit that is also the accumulation point of a
        sequence of pairwise incomparable ends, then $\Homeo(\S)$ and
        $\Map(\S)$ fails to have automatic
        continuity. 
\end{RMK}

Notably, all of the discontinuous maps that we make use of have targets that contain isomorphic copies of $\Q$. This is perhaps not a coincidence, and further evidence for the following question of Conner (the question is originally stated for countable codomains, we extend it to separable codomains). 

\begin{Question}[{\cite[Question 5.2]{CC2019} and \cite[Conner's Conjecture]{CPV2021}}] \label{conj:conner}
	If $H$ is a torsion-free separable group that does not contain an isomorphic copy of $\Q$, must every homomorphism from a completely metrizable group to $H$ be continuous? 
\end{Question}

\begin{EX}\label{ex:examples} Here, we construct a few surfaces for which our theorems apply that were not covered by any previous results.
	\begin{enumerate} 
		\item Consider the surface, $M$, with end space two Cantor sets $C_{1} \cup C_{2}$ so that $C_{1} \cap C_{2}$ is exactly a single point. Furthermore, each end in $C_{1}$ is accumulated by genus. See \Cref{fig:monalisa}. Every end of this surface is telescoping and so $\Homeo(M)$ and $\Map(M)$ have the automatic continuity property. The end space $C_{1} \cap C_{2}$ is the prototypical example of a telescoping end. 
		\item Let $L$ be the Loch Ness monster surface, i.e. the surface with a infinite genus and a single end. If $N$ is the connect sum of $L$ with any orientable, boundaryless surface, then $\Homeo(N)$ and $\Map(N)$ fail to have the automatic continuity property. Here $L$ can also be replaced by any surface whose end space has a countable orbit under the action of the homeomorphism group. 
	\end{enumerate}
\end{EX}

\begin{figure}[h]
	  \centering
		    \includegraphics[width=\textwidth]{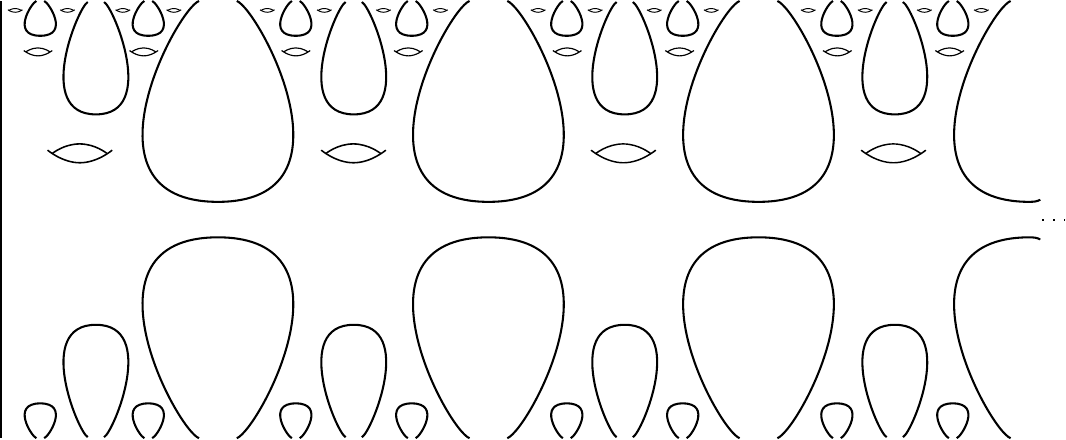}
		    \caption{An example of a surface for which \Cref{thm:surfaceclassification} applies in the positive direction.}
	    \label{fig:monalisa}
\end{figure}

The general techniques for proving automatic continuity developed below also apply to other classes of groups. We apply them to homeomorphism groups of end spaces, i.e., second countable Stone spaces (closed subsets of a Cantor set). Alternatively, by Stone Duality \cite{Stone1936}, these groups can also be thought of as automorphism groups of countable Boolean algebras. 

\begin{customthm}{B}\label{thm:endspaceac}
	Let $X$ be a second countable Stone space. If $X$ is stable, then $\Homeo(X)$ has automatic continuity. 
\end{customthm}

This answers \cite[Question 2.5]{Mann2024} under the
stability hypothesis. In particular, the stability condition can be thought of as a type of ``local homogeneity." Here, stability is defined as in \cite{MR2023} (see \Cref{stable} and \Cref{prop:stable}). We also prove an alternate characterization (\Cref{prop:stable}) of stability that allows us to view each point in $X$ as satisfying a type of ``telescoping" condition. 

By a classical result of Mazurkiewicz and Sierpi\'{n}ski \cite{MS1920}, every countable Stone space is homeomorphic to a countable ordinal of the form $\omega^{\alpha}\cdot n +1$ where $\alpha$ is a countable ordinal and $n \in \N$. Notably, all countable ordinals are stable and so as a corollary we have the following.

\begin{COR}\label{cor:ctbleordinals}
	Let $\omega^{\alpha}\cdot n +1$ be a compact, countable ordinal equipped with the order topology. The group $\Homeo(\omega^{\alpha}\cdot n +1)$ has the automatic continuity property. 
\end{COR}

Concurrent with this work, Hern\'{a}ndez-Hern\'{a}ndez--Hru\v{s}\'{a}k--Rosendal--Valdez \cite{HHRV} have also obtained the previous result. In fact, when $n=1$ they prove that $\Homeo(\omega^{\alpha} + 1)$ has ample generics, a stronger condition that implies automatic continuity. They can then leverage this ample generics property to show that $\Homeo(\omega^{\alpha} \cdot n +1)$ has automatic continuity for arbitrary $n$. 

\Cref{thm:surfaceclassification} and \Cref{cor:ctbleordinals} offer some insight as to how the cases of surfaces and end spaces differ. We have that $\Homeo(\omega^{\alpha}\cdot n +1)$ always has automatic continuity, but if $\S_{\alpha,n}$ is a surface with such an end space, then $\Homeo(\S_{\alpha,n})$ fails to have automatic continuity. Upon reflection, this is perhaps not surprising since the techniques used to prove this failure of automatic continuity are directly coming from the topology of the surface and the mapping class group. In particular, the hyperbolic spaces used in \cite{Domat2022} are built out of curve graphs using techniques from \cite{MM1999, BBF2015}.  

The only examples known to us of Stone spaces whose homeomorphism
groups have discontinuous homomorphisms come from constructions in
\cite{Rosendal2009, Mann2024}. For these examples, one has infinitely
many finite orbits or cardinality larger than one. Notably, the
condition on stability does not allow this to happen. This suggests
the following question towards a classification for general second
countable Stone spaces.

\begin{Question}\label{conj:endspace}
	Does there exist a second countable Stone space $X$ so that $\Homeo(X)$ does not surject onto an infinite product of finite groups and $\Homeo(X)$ fails to have automatic continuity?
\end{Question}

Returning to our original question, ``How does the algebra of the group determine the topology of the group?," an immediate application of the automatic continuity property is to say that the topology on the group is essentially unique. 

\begin{COR}\label{cor:uniquePolishSurfaces}
	Let $\S$ be a stable, orientable surface without boundary, all of whose ends are telescoping. The groups $\Homeo(\S)$ and $\Map(\S)$ both have unique Polish group topologies. 
\end{COR}

\begin{COR}\label{cor:uniquePolishEnds}
	Let $X$ be a stable, second countable, Stone space. The group $\Homeo(X)$ has a unique Polish group topology. 
\end{COR}

This follows from some standard arguments in descriptive set theory, e.g. see \cite[Section 2.5]{Vlamis2024}. 

\subsection*{Acknowledgments}
	The authors thank Kathryn Mann, Nick Vlamis, Jes\'{u}s Hern\'{a}ndez--Hern\'{a}ndez, Ferr\'{a}n Valdez, and Michael Hru\u{s}\'{a}k for several helpful conversations. The first author was supported by NSF DMS-2304774. The second author was partially supported by NSF DMS–1745670, NSF DMS-2303262, and the Fields Institute for Research in Mathematical Sciences. The third author was supported by NSERC Discovery Grant RGPIN-5507.

\subsection{The Five Step Program}\label{ssec:5step}

	We first provide a general outline for how to prove the automatic continuity property for a Polish group of homeomorphisms. We will follow this general outline in \Cref{sec:endspaces} and \Cref{sec:surfaces} for Stone spaces and surfaces, respectively. We remark that this general outline is inspired by the arguments found in \cite{RS2007, Rosendal2008, Mann2016, Mann2024}. We expect this general framework to be applicable to other spaces and Polish groups of transformations. In particular, our arguments should be applicable for subgroups of homeomorphisms provided they contain ``enough" maps of an appropriate type.

For all of our proofs of automatic continuity, we will actually prove that the Steinhaus property holds and apply a result of Rosendal-Solecki \cite{RS2007}.

\begin{DEF}
    A topological group $G$ has the \textbf{Steinhaus property} if there exists an $n \in \N$ such that, whenever $W$ is a symmetric subset of $G$ such that countably many translates of $W$ cover $G$, $W^{n}$ contains an open neighborhood of the identity in $G$. 
\end{DEF}

\begin{PROP}\cite[Proposition 2]{RS2007}
    If a topological group $G$ has the Steinhaus property, then it also has the automatic continuity property. 
\end{PROP}

Throughout this section we will give vague definitions that will each
be modified to fit the subsequent sections. All of our proofs require
some sort of \textbf{locally telescoping} condition. This condition
can be thought of as a type of strong local homogeneity of the
space. In particular, it will allow us to decompose a neighborhood,
$\Omega$, of a point into an uncountable collection of
\textbf{bricks}. These bricks will all be pairwise homeomorphic via
homeomorphisms supported only on $\Omega$ and the complement of each brick
must again be a brick. Furthermore, each brick must itself contain uncountably many bricks. This may give some insight
into why we require ends to be of Cantor type in
\Cref{thm:surfaceclassification}.

\begin{EX} Here we give three examples of bricks. The first is in a Stone space and the second two are in surfaces.
	\begin{enumerate}
		\item Let $\displaystyle X = \left\{\frac{1}{n}\bigg | n \in \N\right\} \cup \{0\} \subset \R$ equipped with subspace topology. A brick of $X$ is then a countably infinite subset of points (not including $0$) so that its complement is also countably infinite. Such a subset is sometimes referred to as a \emph{moiety} in the literature. Here $\Homeo(X)$ is isomorphic to the symmetric group on $\N$.
		\item Let $\S_{1}$ be a punctured sphere. Then a brick about the puncture, $p$, of $\S_{1}$ is a locally finite collection of disjoint annuli $\{A_{i}\}$ such that each $A_{i}$ separates $p$ from $A_{j}$ for all $j<i$.
		\item Let $M$ be the surface described in the \Cref{ex:examples} above and let $e = C_{1} \cap C_{2}$. Then we say a big annulus is a subsurface $A$ so that
		\begin{itemize}
			\item $A$ is bounded by two curves, $\gamma_{1}$ and $\gamma_{2}$, so that $\gamma_{2}$ is contained in the complementary component of $\gamma_{1}$ witnessing the end e, and 
			\item the end space of $A$ intersects both $C_{1}$ and $C_{2}$ non-trivially. 
		\end{itemize}
		We then define a brick to be a countably infinite, locally finite, disjoint union of big annuli. See \Cref{fig:monalisaannuli}.
	\end{enumerate}
\end{EX}

\begin{figure}[h]
	  \centering
		    \includegraphics[width=\textwidth]{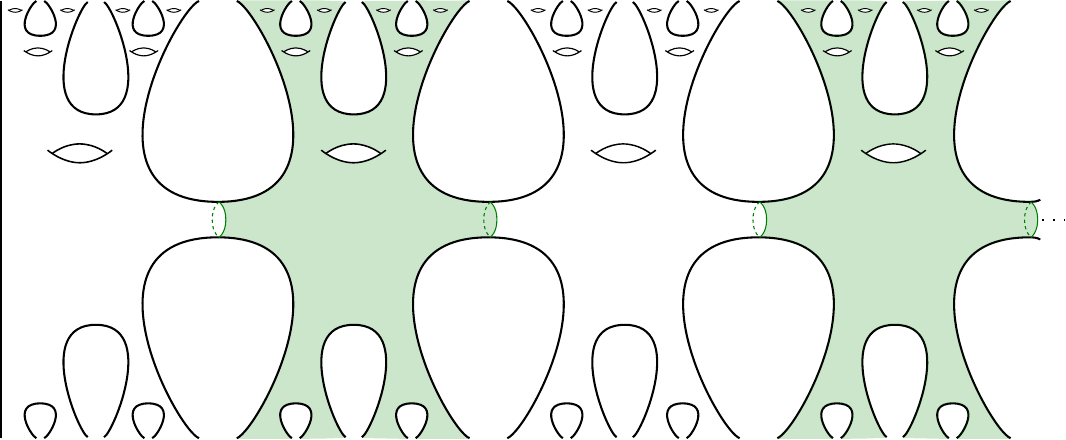}
		    \caption{An example of a brick (in green) comprised of big annuli in a surface.}
	    \label{fig:monalisaannuli}
\end{figure}

Now let $G = \Homeo(X)$ for some compact space $X$ and $W$ be a symmetric set in $G$ so that countably many translates of $W$ cover $G$. First we will always apply a standard Baire category argument (\Cref{lem:baire}) to find an open neighborhood of the identity $U \subset G$ so that $W^{2}$ is dense in it. This $U$ will determine some finite partition of $X$ into $\Omega_{1}\sqcup \Omega_{2} \sqcup \cdots \sqcup \Omega_{m}$ so that any $g \in G$ preserving each of the $\Omega_{i}$ belongs to $U$. Next we use compactness to further partition each $\Omega_{i}$ into telescoping neighborhoods. This reduces the problem of proving Steinhaus to the following claim. We let $G(\Omega)$ denote the subgroup of $G$ consisting of homeomorphisms supported on $\Omega\subset X$. 

\begin{CLAIM}\label{claim:telescoping}
	Let $\Omega$ be telescoping with respect to $y$. There exists an $N\geq 0$ so that if $W \subset G$ is symmetric, with countable many translates covering $G(\Omega)$, and $W^{2}$ dense in $G(\Omega)$, then $G(\Omega) \subset W^{N}$.
\end{CLAIM}

Proving this claim will proceed through several standard steps.

\textbf{Step 1: Fragmentation} First we will fragment any $g \in G(\Omega)$ into two maps, $g=g_{1}g_{2}$ so that each $g_{i}$ fixes a point $y_{i} \in \Omega$ and so that $\Omega$ is telescoping with respect to $y_{i}$. Next we fragment again to write such a map as a product of two maps supported on bricks. The final outcome is that given any $g \in G$ we can write $g=g_{1}g_{2}g_{3}g_{4}$ with each $g_{i}$ supported on a brick in $\Omega$. This reduces the problem to finding some $N$ so that any map supported on a brick is in $W^{N}$. 

\textbf{Step 2: Finding Commutators} Next we want to realize any map supported on a brick as a uniformly finite product of commutators, each of which is supported on a single brick. This step will always proceed by using a version of the Anderson Trick \cite{Anderson1958}. How exactly it is implemented will vary depending on our situation and the exact form of the ``telescoping condition'' above. This will depend on what types of shift maps exist in $G$. E.g., in the case of a 0-dimensional space, we have a well-defined shift map supported on a single brick that we can use to write any map as a single commutator. However, in the surface case we will only have shift maps supported in each connected component of a brick. In this case we will have to fragment again using an Eilenberg-Mazur swindle \cite{Bass1963,Mazur1959}, taking advantage of the existence of certain infinite products of maps. 

\textbf{Step 3: Diagonal Argument} We now use the above in order to find one specific ``good'' brick $\cB$ in $\Omega$ so that any $g$ supported on $\cB$ is in $W^{N}$. This part of the argument will be fairly standard throughout. In fact, the arguments in this step are taken almost directly from \cite{Mann2016, Mann2024,Rosendal2008,RS2007}.

\textbf{Step 4: Pigeonhole} Next we run a pigeonhole argument to find elements of $W^{N}$ that move any brick into our preferred brick $\cB$. This will make use of some part of the ``telescoping condition'' that guarantees that we can map any brick onto any other brick. This will again proceed in a fairly standard fashion in each case following \cite{Mann2024}. 

\textbf{Step 5: Wrapping Up} Finally, we put all of the above pieces together to prove \Cref{claim:telescoping}. Then applying a Baire category argument, \Cref{lem:baire}, we upgrade this to obtain Steinhaus, and thus automatic continuity, for $G(\Omega)$ whenever $\Omega$ is a telescoping neighborhood.

\section{Background on Automatic Continuity}\label{sec:background}

\begin{DEF}
    A topological group $G$ has the \textbf{automatic continuity property (AC)} if every homomorphism from $G$ to a separable group is continuous.
\end{DEF}

Using an induction construction we can see that if a closed countable-index subgroup fails to have AC, then so does the larger group. See also the introduction of \cite{Rosendal2009} for a more direct proof of this.  Suppose $G$ is a group and $H<G$ a subgroup. Suppose also that $H$ acts on the left on a set $X$. The goal is to construct a larger set $Y$ on which $G$ acts, extending the action of $H$ on $X$. This will not be a literal extension, but it will be close enough.  

Start with the action of $G$ on $G \times X$ by $g\cdot (g',x) = (gg',x)$. Define the following equivalence relation on $G \times X$: 
\begin{align*}
    (g,x) \sim (gh,h^{-1}x) \text{ for all } h \in H.
\end{align*}
Let $Y$ be the quotient set. Note that:
\begin{itemize}
    \item $G$ acts on $Y$ via $g\cdot [g',x] = [gg',x]$.
    \item If $\{g_{i}\}$ is a set coset representatives so that $G = \bigsqcup_{i \in \N} g_{i}H$, then 
    \begin{align*}
        Y = \bigsqcup_{i \in \N} \{[g_{i},x]\vert x \in X\}.
    \end{align*}
    Moreover, we can canonically identify $\{[1,x]\vert x \in X\}$ with $X$ via the map $[1,x] \mapsto x$ and under this identification we have $\{[g_{i},x]\vert x \in X\} = g_{i}X$ so that $Y =\bigsqcup_{i\in\N} g_{i}X$. 
\end{itemize}

\begin{LEM} \label{lem:induction}
    Let $G$ be a topological group, $H < G$ a countable index subgroup, and $f:H \rightarrow A$ a discontinuous homomorphism to a countable discrete group. Then there is a separable topological group $B$ (isomorphic to $S_{\omega} = \Homeo(\omega+1)$) and a discontinuous homomorphism $F:G \rightarrow B$. 
\end{LEM}

\begin{proof}
    Take $X = A$ equipped with the discrete topology. The subgroup $H$ acts via left multiplication on $X$ via $f$, i.e. $h\cdot x = f(h)x$. Let $Y \supset X$ be as above with the induced action of $G$. Thus we have a homomorphism $F:G \rightarrow \Aut(Y) =: B$, the group of bijections of $Y$. Note first that $Y$ is countable. Indeed, $Y =\bigsqcup_{i\in\N} g_{i}X$, a countable union of countable sets. Thus the group of bijections $\Aut(Y)$ with the permutation topology is isomorphic to $S_{\omega}$. It remains to show that $F$ is discontinuous. It suffices to argue that the restriction to $H$ is discontinuous. The restricted action preserves $X \subset Y$, so it suffices to argue that $H \rightarrow \Aut(X)$ is discontinuous. The image of this homomorphism is contained in the group of left translations of $X = A$, which can be identified with $A$. In fact, the image of this homomorphism can be exactly identified with the image of the original $f:H \rightarrow A$. Thus we see that the homomorphism $H \rightarrow \Aut(X)$ is discontinuous by assumption. 
\end{proof}

The starting point of all of our arguments will be the following standard application of the Baire Category theorem.

\begin{LEM} \label{lem:baire}
	Let $G$ be a Polish group. If $W$ is a symmetric subset of $G$ such that countably many translates of $W$ cover $G$, then the closure of $W^{2}$ contains a neighborhood of the identity in $G$. 
\end{LEM}

\begin{proof}
	 Write $G = \bigcup g_{i}W$. By Baire category, some $\overline{g_iW}$ has nonempty
  interior. Choose some $g_iw\in g_iW$ that belongs to this
  interior. Then the closure of $w^{-1}g_i^{-1}\cdot
  g_iW=w^{-1}W\subseteq W^2$
  contains a neighborhood of the identity.
\end{proof}

\section{End Spaces}\label{sec:endspaces}

Let $X$ be a compact totally disconnected metrizable space (i.e., a second countable Stone space). 
In this section we investigate whether the group $\Homeo(X)$ has  the automatic continuity property. More
generally, we fix a collection $\{X_\alpha\}_{\alpha\in\mathcal A}$ of
closed subsets of $X$ and we consider the group
$\Homeo(X,\{X_\alpha\}_{\alpha\in\mathcal A})$ of homeomorphisms that
preserve each $X_\alpha$ (setwise). We think of each
$\alpha\in\mathcal A$ as a color, so elements of
$\Homeo(X,\{X_\alpha\}_{\alpha\in\mathcal A})$ are color preserving
homeomorphisms of $X$. To simplify notation we omit the colors and
talk about $\Homeo(X)$, but coloring is understood. 
That is, when we say $\phi : U \to V$ is a homeomorphism from 
$U \subset X$ to $V \subset X$ we always assume that, for every 
$\alpha \in \mathcal A$, 
\[
\phi(U \cap X_\alpha) = V \cap X_\alpha.
\] 
The coloring is motivated by Section~\ref{sec:surfaces} where we need to distinguish 
between planar and non-planar ends of a surface. In the first reading, 
the reader may think about the case $\mathcal
A=\emptyset$.
The goal of this section is to prove the following theorem. The
definition of stability is below in Definition \ref{stable}.

\begin{THM}\label{thm:endsmain}
	Let $X$ be a second countable Stone space. If $X$ is stable
        (relative to the given colors), then $\Homeo(X)$ has the Steinhaus property and therefore the automatic continuity property.
\end{THM}

By a neighborhood $U$ in $X$ we always mean a clopen neighborhood, that is a set that is both open and closed. 
We say a family of subsets $Y_k \subset X$ descends to $x$ and write $Y_k \searrow x$, if for every sequence $x_k \in Y_k$, we have $x_k \to x$.

\begin{PROP} \label{prop:stable}
Let $X$ be a second countable Stone space, $x$ be a point in $X$
and let $U$ be a neighborhood of $x$. Then the following are equivalent. 
\begin{enumerate}[(i)]
\item For every neighborhood $U' \subset U$ of $x$ there is a homeomorphism 
$\Phi : U \to U'$ fixing $x$. 
\item For every neighborhood $U' \subset U$ there is a neighborhood $U'' \subset U'$ of $x$
and a homeomorphism $\Phi : U \to U''$ fixing $x$. 
\item There is a decomposition 
\begin{equation}  \label{Eq:Decomposition-U}
U-\{x\} = \bigsqcup_{k=1}^\infty Y_k
\end{equation} 
where $Y_k$ are clopen, disjoint, $Y_k \searrow x$ and, for every $k \geq 1$, 
the set $Y_{k+1}$ contains a homeomorphic copy of $Y_k$. Furthermore, for every sequence of indices 
$k_n$, 
\begin{equation}  \label{Eq:sub-sequence} 
U - \{x\} \homeo \bigsqcup_{n=1}^\infty Y_{k_n},
\end{equation} 
\end{enumerate} 
\end{PROP} 

\begin{proof} 
The implication $(i) \Longrightarrow (ii)$ is immediate. We start by proving $(ii) \Longrightarrow (iii)$.
Let 
\[
U = U_0 \supset U_1 \supset U_2 \supset \dots
\] 
be a nested sequence of clopen neighborhoods of $x$ such that 
$U_i \searrow x$ and there are homeomorphisms 
\[
\Psi_i: U_i \to U_{i+1}, \qquad i=0, 1, 2, \cdots
\] 
fixing $x$. This is possible by starting from a nested sequence that descends to $x$ and 
making them smaller to make sure they are all homeomorphic to $U$. 
We construct disjoint clopen sets $V_k$, $Y_k$ and an index sequence $p_k$ inductively as follows.  
Set $p_0 = 0$, $p_1 =1$ and set 
\[
Y_1=V_1 = U_0 - U_1.
\] 
Assume now that the index $p_k>0$ and sets $V_k, Y_k \subset U-\{ x\}$ are given. Let 
\[
\Phi_k : U_{p_{k-1}} \to U_{p_k}
\qquad\text{be the composition} \qquad
\Phi_k  \defeq\Psi_{(p_{k}-1)} \circ \dots \circ \Psi_{(p_{k-1}+1)} \circ \Psi_{p_{k-1}}.
\]
Then $\Phi_k (Y_k)$ is a closed subset of $U_{p_k}$ disjoint from $x$. Hence, 
we can choose an index $p_{k+1}>p_k$ large enough such that 
\[
\Phi_k(Y_k) \cap U_{p_{k+1}} = \emptyset.
\]
Then define 
\[
Y_{k+1} = U_{p_k}- U_{p_{k+1}} 
\qquad\text{and}\qquad
V_{k+1} = Y_{k+1} -\Phi_k(Y_k). 
\]
The sets $Y_k$ are clopen and disjoint. Since $Y_{k+1} \subset U_{p_k}$ and $U_i \searrow x$, 
we also have $Y_k \searrow x$. Also, 
\begin{equation} \label{Eq:W}
 U-\{x\} = \bigsqcup_{k=1}^\infty (U_{p_k} - U_{p_{k+1}}) = \bigsqcup_{k=1}^\infty Y_k.
\end{equation} 

To see the last two assertions of $(iii)$, note that $Y_{k+1}$ is homeomorphic to $Y_k \sqcup V_{k+1}$. 
By induction, this implies 
\[
Y_k \homeo \bigsqcup_{j=1}^k V_j. 
\] 
Now, let a sequence of indices $k_n$ be given and let 
$A= \bigsqcup_{n=1}^\infty Y_{k_n}$. Since a copy of $V_j$ appears in 
$Y_k$ for every $k \geq j$, we can break this further and write 
\[
A= \bigsqcup_{n=1}^\infty \bigsqcup_{j=1}^{k_n} V_{n, j} 
\]
where $V_{n,j}$ is a homeomorphic copy of $V_j$ in $Y_{k_n}$. Changing the order of unions, we have 
\[
A = \bigsqcup_{j=1}^\infty \bigsqcup_{n=1}^{\infty} V_{n, j}.  
\]
Similarly, using \eqref{Eq:W}, we have 
\[
U-\{x\}= \bigsqcup_{k=1}^\infty Y_k = \bigsqcup_{k=1}^\infty \bigsqcup_{j=1}^{k} V_{k, j} 
= \bigsqcup_{j=1}^\infty \bigsqcup_{k=1}^{\infty} V_{k, j},
\]
where $V_{k,j}$ is a homeomorphic copy of $V_j$. We can now construct a homeomorphism from $A$ to 
$U-\{x\}$ by sending $V_{k,j}$ homeomorphically to $V_{n,j}$. This map is continuous since, for every $j$, 
\[
V_{n,j}\searrow x \qquad\text{and}\qquad V_{k,j}\searrow x
\qquad\text{as $n \to \infty$}. 
\]
This finishes the proof of $(ii) \Longrightarrow (iii)$. 

To see $(iii) \Longrightarrow (ii)$, we observe that, for every $U' \subset U$, there is $N>0$ such that, 
$Y_k \subset U'$ for $k \geq N$. Otherwise, there is a point $x_k \in Y_k-U$. But $x_n$ does not limit 
to $x$ which contradicts the assumption that $Y_n \searrow x$. Hence, 
\[
U \homeo \bigsqcup_{k=N}^\infty Y_k \subset U'.
\]

It remains to show $(ii) \Longrightarrow (i)$. Let $U' \subset U$ be given.
Let $U=U_1 \supset U_2 \supset \dots$ and $\Psi_i : U_i \to U_{i+1}$ 
be as before (which exist after assuming $(ii)$). We construct disjoint clopen sets $W_k$ and an index sequence $q_k$ as follows.
Set $q_1 =1$ and set $W_1 = U - U'$. Assuming $q_k>0$ and $W_k \subset U-\{x\}$ are given,
let 
\[
\Phi_k' : U_{q_{k-1}} \to U_{q_k}
\qquad\text{be the composition} \qquad
\Phi_k'  \defeq\Psi_{(q_{k}-1)} \circ \dots \circ \Psi_{(q_{k-1}+1)} \circ \Psi_{q_{k-1}}
\]
and choose an index $q_{k+1} > q_k$  large enough such that 
\[
\Phi_k'(W_k) \cap U_{q_{k+1}} = \emptyset.
\]
Then define $W_{k+1} = \Phi_k(W_k)$. The sets $W_k$ are all homeomorphic, disjoint
and descend to $x$. We can now build a homeomorphism from $U \to U'$
by sending $W_k$ to $W_{k+1}$ and fixing everything else. This map is continuous
since $W_k \searrow x$. This finishes the proof. 
\end{proof} 

\begin{DEF}\cite[Definition 4.14]{MR2023}\label{stable}
	A neighborhood $U$ of a point $x \in X$ is \textbf{stable} if it satisfies the three equivalent conditions in \Cref{prop:stable}. We say that a point in $X$ is \textbf{stable} if it has a stable neighborhood. We say the space $X$ is \textbf{stable} if every point of $X$ is stable.
\end{DEF}

Property $(iii)$ of stability will play the role of a type of
``telescoping'' property for second countable Stone spaces.

Before proceeding with the proof of \Cref{thm:endsmain} we give some examples.

  \begin{example} Here are two examples of stable Stone spaces.
  \begin{enumerate}[(1)]
  \item The Cantor set is stable and thus its homeomorphism
    group has
    AC. Our theorem thus recovers the result of \cite{RS2007}. 
    
  \item Every countable compact Stone space is stable and
    hence has automatic continuity. This follows from the classification theorem that realizes compact Stone spaces as countable ordinals \cite{MS1920}. 
    \end{enumerate}
  \end{example}

  \begin{remark}
  Adding a coloring to the space may change whether the space is
  stable. For example, let $X=C\times D$ where $C$ is the Cantor set
  and $D=\{0\}\cup\{\frac 1n\mid n\in\Z_{n\geq 1}\}$. Then this space
  is stable and hence $\Homeo(X)$ has the automatic continuity
  property. However, if we now color each $C\times\{d\}$ with its own
  color for $d\in D$, the space is no longer stable and we do not know
  whether the color-preserving homeomorphism group has automatic
  continuity. This is an end space analog of the surface example given in \Cref{ssec:unknownex}.
  \end{remark}
  
Unless otherwise stated, throughout the following we assume that $X$
is a second countable Stone space, $x\in X$, and $\Omega\subset X$ is
a stable neighborhood of $x$.

We have a decomposition 
\begin{equation} \label{Eq:Decomposition}
\O \setminus \{x\} = \bigsqcup_{k=1}^{\infty} Y_{k}
\end{equation} 
as given by \Cref{prop:stable}. 

Denote 
\[
Y[n]:=\bigsqcup_{i=1}^n Y_i, 
\qquad\text{and more generally}\qquad 
Y[n,m]=\bigsqcup_{i=n}^m Y_i.
\]

\begin{PROP}\label{prop:axiomstar}
  For every $1\leq n<m$ and every embedding
  $\phi:Y[n]\to Y[m]$ onto a clopen subset, there is some $p>m$ and a
  homeomorphism $h:Y[p]\to Y[p]$ that agrees with $\phi$ on $Y[n]$.
\end{PROP}

The proof is based on the following lemma.

\begin{lemma}\label{lem:puzzle}
  Let $A,B$ be topological spaces and let $\phi:A\to A\sqcup B$ be an embedding onto a
  clopen subset of $A \sqcup B$. Then there is a self-homeomorphism
  $\tilde\phi$ 
  of $A\sqcup B\sqcup A$ that restricts on (the first) $A$ to
  $\phi$.
\end{lemma}

\begin{proof}
Since $\phi(A)$ is clopen in $A\sqcup B$,  we have $A \sqcup B \homeo \phi(A)\sqcup C$. We
  define $$\tilde\phi:A\sqcup B\sqcup A\to \phi(A)\sqcup
  C\sqcup A$$ to be $\phi$ on $A$ and a homeomorphism from $B\sqcup
  A$ to $C\sqcup A$.
\end{proof}

\begin{proof}[Proof of Proposition \ref{prop:axiomstar}]
  Take $A=Y[n]$, $B=Y[n+1,m]$. Then $Y[m+1,2m]$ contains a clopen copy
  of $A \sqcup B$. Now
  apply the lemma and extend the resulting homeomorphism by the
  identity to all of $Y[2m]$.
\end{proof}

In this setting, we say that a \textbf{brick}, $\cB$, is an infinite union of the $Y_{k}$ so that the complement in $X \setminus \{x\}$ is also an infinite union of $Y_{j}$. A \textbf{subbrick}, $\cC \subset \cB$, of a brick $\cB$ is a brick $\cC$ such that $\cB \setminus \cC$ is also a brick. Note that the furthermore statement of property $(iii)$ from \Cref{prop:stable} implies that all bricks are homeomorphic. We will need the following proposition on the existence of shift maps on bricks. 

\begin{LEM}\label{lem:shift}
  If $\O$ is a stable neighborhood with respect to $x$, the $Y_{j}$ are as in \eqref{Eq:Decomposition} and $\cB \subset \O$ is a brick, then
  there is a homeomorphism $\sigma:\O\to
  \O$ fixing $x$ such that the collection of $\sigma^i(\cB)$ are
  pairwise disjoint bricks for $i\in\Z$. Furthermore, the $\sigma$--orbit of every point in $\O - \{x\}$ accumulates to $x$. 
\end{LEM}

\begin{proof}
  Using a bijection between $\N$ and $\Z^2$, write
  \[
  \O\smallsetminus\{x\}=\bigsqcup_{\kappa\in\Z^2}Y_\kappa
  \qquad\text{so that}\qquad  
  \cB=\bigsqcup_{\kappa\in\Z\times\{0\}}Y_\kappa.
  \] 
  Note that, for every $m \in Z$, $\bigsqcup_{\kappa\in\Z\times\{m\}}Y_\kappa$ is a brick.
  Since all bricks are homeomorphic, there exists a homeomorphism $\sigma : \O \to \O$ so that $\sigma$ sends 
  \[
  \bigsqcup_{\kappa\in\Z\times\{m\}}Y_\kappa
  \qquad\text{homeomorphically to}\qquad
  \bigsqcup_{\kappa \in\Z\times\{m+1\}}Y_\kappa
  \] 
  for all $m\in\Z$. 
\end{proof}

We now start the five step program with the assumption that $\O$ is a
stable neighborhood of the point $x \in \O$.

\begin{assumptions} We first set up some notation and standing assumptions. Unless explicitly stated, these will hold until the end of \Cref{sec:endspaces}.
\begin{enumerate}[(1)]
\item Denote by $G(\O)$ the
subgroup of $G=\Homeo(X)$ consisting of homeomorphisms supported on
$\O$; thus $G(\O)$ can be naturally identified with $\Homeo(\O)$.  
\item Fix the decomposition $\Omega \setminus \{x\} = \bigsqcup_{k=1}^{\infty}Y_{k}$ from \eqref{Eq:Decomposition}. 
\item Assume that $W <G$ is a symmetric set so that countably many translates cover $G$ and
$W^{2}$ is dense in $G(\O)$ (i.e. $\overline{W^2}\supset G(\O)$). 
\item For any brick, $\cB$, we write $G(\cB)$ for the closed subgroup of
homeomorphisms of $\O$ or $X$ supported on $\cB \cup
\{x\}$.
  \end{enumerate} 
\end{assumptions}

\subsection{Step 1: Fragmentation}

We first fragment an element of $G(\O)$ into two maps that have fixed points in the orbit of $x$. 

\begin{lemma}\label{lem:fpfragment}
	Any $h \in G(\O)$ can be written as $h=h_{1}h_{2}$ so that
        each $h_{i}$ belongs to $G(\O)$ and fixes some point in the
        orbit $G(\O) \cdot x$. 
\end{lemma}

\begin{proof}    
    Suppose $h(x)\neq x$ and assume $h^{-1}(x)\in Y_k$ for some
    $k$. By \Cref{lem:shift}, the orbit of $x$
    accumulates on $x$. Hence, there exists $j\neq k$ so that $Y_j$
    intersects $G(\O)\cdot x$.  Note that $h(Y_j)$ is disjoint from a
    neighborhood of $x$, so it is contained as a clopen subset of
    some finite union $Y_{k_1}\cup\cdots\cup Y_{k_r}$. Using
    \Cref{prop:axiomstar} we can find a homeomorphism $h_1$ of
    $X$ supported on only finitely many $Y_s$ and agreeing with $h$ on
    $Y_j$ (to apply \Cref{prop:axiomstar} first extend
    $h|_{Y_j}$ to $\bigsqcup_{i\leq j}Y_i$ by sending $\bigsqcup_{i< j}Y_i$
    arbitrarily to a clopen set disjoint from $h(Y_j)$). Thus
    $h=h_1h_2$ where $h_1$ fixes $x$ and $h_2=h_1^{-1}h$ is identity
    on $Y_j$ and so fixes a point in $G\cdot x$.
\end{proof}

Note that $\O$ is a stable neighborhood with respect to any point in the
    $G(\O)$-orbit of $x$. 
Thus the
previous lemma allows us to reduce our problem to only considering
maps in $\Homeo(\O,x)$. Next we fragment again, this time into
bricks.

\begin{lemma}\label{lem:brickfragment}
  Any $f\in \Homeo(\O,x)$ can be written as $f=g h$ so that both
  $g,h\in \Homeo(\O,x)$ are 
  supported on the closure of a brick.
\end{lemma}

\begin{proof}
Choose indices $n_k$ and $m_k$, where $m_k \in [n_{k-1}+1, n_k-1]$ inductively as follows. Let $n_0 = 0$ 
and $m_1=1$. Then assuming $n_1, \dots, n_{k-1}$ and $m_1, \dots, m_k$ are given
,  choose $n_k > m_k$ 
large enough so that 
\[
f^{\pm 1}(Y_{m_k}) \subset \bigsqcup_{i=n_{k-1}+1}^{n_k-1} Y_i 
\] 
and then choose $m_{k+1}$ large enough such that 
\[
f^{\pm 1}(Y_{m_{k+1}}) \cap \left( \bigsqcup_{i=1}^{n_k} Y_i \right) = \emptyset
. 
\]
Define
\[
\cA_m =  \bigsqcup_{k=1}^\infty  Y_{m_k} 
\qquad\text{and}\qquad
\cA_n =  \bigsqcup_{k=1}^\infty  Y_{n_k} 
\]

Let $\cA_m'$ and $\cA_n'$ be complementary bricks to $\cA_m$ and $\cA_n$ respectively. 
Note that $f(\cA_m)$ is disjoint from $\cA_n$. Therefore, there is a
homeomorphism $g:X \to X$  
with support in $\cA_n'$ such that 
$f|_{\cA_m} = g|_{\cA_m}$. In particular, $h=g^{-1}f$ preserves the brick $\cA_m$.
But $g$
preserves the brick 
$\cA_n$ and $f = gh$. This finishes the proof. 
\end{proof}

Thus, combining these two lemmas we see that any map $h \in G(\O)$ can
be fragmented as $h = h_{1}h_{2}h_{3}h_{4}$ where each $h_{i}$ is in
$G(\O)$ and it is supported on a brick (although not necessarily centered around the same point). Our new goal is to show that any map supported on a brick satisfies a Steinhaus condition. 

\subsection{Step 2: Finding Commutators}

\begin{LEM}\label{lem:anderson}
	Let $\cA$ be a brick around $x$ and $\cB$ a subbrick of $\cA$. For every $h \in G(\cB)$, there exists some $u,v \in G(\cA)$ so that $h=[u,v]$. 
\end{LEM}

\begin{proof}
	We apply \Cref{lem:shift} to find a homeomorphism $v\in
        \Homeo(\O,x)$ supported on $\cA$ so that $v^{i}(\cB)$ are
        pairwise disjoint for $i \in \Z$. Note that we can apply this
        lemma since any brick is itself homeomorphic to $\O \setminus
        \{x\}$. Now let $u$ be the map that is exactly a copy of $h$
        on each $v^{i}(\cB)$ for $i \geq 0$, i.e. $u =
        \prod_{i=0}^{\infty} vhv^{-1}$. Then we have $h = [u,v]$ as
        desired.
\end{proof}

\subsection{Step 3: Diagonalization}

Next we will find a ``good'' subbrick. Namely, any map supported on this subbrick will be in $W^{8}$. First we see that we can find a subbrick on which we can approximate maps. A similar step is carried out in \cite{RS2007, Rosendal2008, Mann2016, Mann2024}.

\begin{lemma}\label{lem:diagonal}
  Let $\cA$ be a brick and $\cA_1,\cA_2,\cdots$ pairwise disjoint
  subbricks. Then
  \begin{enumerate}[(i)]
    \item there is some $i\geq 1$ such that any homeomorphism of $\cA_i$
      extends to a homeomorphism of $X$ supported on $\cA$ that belongs to $g_iW$, and
      \item moreover, every homeomorphism of $\cA_i$
        extends to a homeomorphism of $X$ supported on $\cA$ that belongs to $W^2$.
  \end{enumerate}
\end{lemma}

\begin{proof}
  We prove (i) by contradiction. Suppose for every $i$ there is a
  homeomorphism $h_i$ of $\cA_i$ that does not extend to a homeomorphism
  of $X$ supported on $\cA$ that belongs to $g_iW$. Define $h$ to be a homeomorphism of
  $\cup_i \cA_i$ that agrees with $h_i$ on $\cA_i$.
  Extend $h$ by the identity on $X\smallsetminus \cA$. Then
  $h\in g_iW$ for some $i$, contradicting the choice of $h_i$. For
  (ii), let $\tilde h\in G(\cA)$ be a homeomorphism in $g_iW$ extending the
  given homeomorphism $h$ of $\cA_i$, and let $\tilde f\in g_iW\cap G(\cA)$ extend the
  identity on $\cA_i$. Then $\tilde f^{-1}\tilde h\in W^2$ extends $h$.
\end{proof}

Next we upgrade this approximation.

\begin{lemma}\label{lem:diagonal2}
	Let $\cA$ be a brick. Then there exists a subbrick $\cZ$ of $\cA$ so that $G(\cZ) \subset W^{8}$. 
\end{lemma}

\begin{proof}
Apply Lemma \ref{lem:diagonal} to $\cA$ and label the resulting subbrick $\cB$. Repeat this process to $\cB$ and label the subbrick $\cC$. 
Let $\cZ$ be any subbrick of $\cC$. Now apply \Cref{lem:anderson} to $\cZ$ as a subbrick of $\cC$. Therefore, for $h \in G(\cZ)$, 
there exists some $u,v \in G(\cC)$ so that $h = [u,v]$. 
	
The support of $u$ is contained in $\cC$ and so we can apply Lemma \ref{lem:diagonal}  to find a $\bar{u}$ such that $\bar{u} \in W^{2}$, $\bar{u}$ is supported on 
$\cB$, and $\bar{u}\vert_{\cC} = u$. 
Similarly, the support of $v$ is contained in $\cC \subset \cB$. Lemma \ref{lem:diagonal} implies that there exists $\bar{v}$ such that $\bar{v} \in W^{2}$, 
$\bar{v}$ is supported on $\cA$, and $\bar{v}\vert_{\cB}= v$. In particular $\bar{v}$ is the identity in $\cB - \cC$. 

Since $\bar{u}$ is the identity in $X - \cB$ and $\bar{v}$ is the identity in $\cB - \cC$, we have $[\bar{u},\bar{v}]$ is the identity in $X-\cC$. 
Thus we have $[\bar{u},\bar{v}]= [u,v] = h$ and hence $h \in W^{8}$.
\end{proof}

\subsection{Step 4: Pigeonhole}

In the previous step we found a ``good'' subbrick. Next we will see that we can conjugate any map supported on any given brick by maps in $W$ in order to have support on this good brick (up to a finite amount of error). This will allow us to prove that any map supported on a brick is in $W^{24}$. 

\begin{lemma}\label{lem:pigeonhole}
	Every $h \in \Homeo(\O,x)$ supported on a brick is in $W^{24}$. 
\end{lemma}

\begin{proof} Suppose $h$ is supported on a brick $\cA=\bigsqcup_i A_i$. 
  Let $\cZ$ be as in Lemma \ref{lem:diagonal2} applied to the complementary brick to $\cA$, so that $\cZ$ and
  $\cA$ are disjoint.
  Write $\cZ=\sqcup_{i=1}^\infty \cZ_i$, where each $\cZ_{i}$ is itself a brick. Fix uncountably many infinite subsets
  $\Lambda_\alpha\subset\mathbb N$ that pairwise intersect in finite
  sets, and set $\cZ_\alpha=\sqcup_{i\in\Lambda_\alpha} \cZ_i$. Fix a
  homeomorphism $f_\alpha:X\to X$ that fixes $x$ and interchanges
  $\cZ_\alpha$ and its complement in $X\smallsetminus\{x\}$. By the
  pigeon-hole principle, there exist $\alpha\neq\beta$ so that
  $f_\alpha,f_\beta\in g_iW$ for some $i$, and in particular
  $f_\beta^{-1}f_\alpha\in W^2$. We claim that
  $F:=(f_\beta^{-1}f_\alpha)h(f_\beta^{-1}f_\alpha)^{-1}$ is supported on
  $\cZ$ plus a finite number of the $A_i$.
  
  To prove the claim, first note that the support of $F$ is contained
  in $\cup_{i=1}^\infty f_\beta^{-1}f_\alpha(A_i)$. Recall that $\cA$ is disjoint from $\cZ$ and hence from $\cZ_\alpha$, so
  $f_\alpha(A_i)$ is contained in $\cZ_\alpha$ for all $i$. Furthermore, if $i$ is sufficiently
  large then $f_\alpha(A_i)$ will be disjoint from $\cZ_\beta$ since
  $\cZ_\alpha\cap \cZ_\beta$ is disjoint from a neighborhood of $x$. It
  follows that for all sufficiently large $i$, $f_\beta^{-1}f_\alpha(A_i)\subset Z_\beta\subset Z$,
  which proves the claim.
  
  Thus we can write $F = F_{1}F_{2}$ where $F_{1}$ is supported on
  finitely many of the $Y_{k}$, $\bigsqcup_{k=1}^{n} Y_{k}$, and
  $F_{2} \in G(\cZ)$. By \Cref{lem:diagonal2} we have that $F_{2} \in
  W^{8}$. It only remains to check that $F_{1}$ is in $W^{12}$. By the
  property $(iii)$ of \Cref{prop:stable}, there is a homeomorphism $g
  \in G(\O)$ that maps $\bigsqcup_{i=1}^{n} Y_{i}$ into $\cZ$. Since
  $W^{2}$ is dense in $G(\O)$, we take $\tilde{g} \in W^{2}$ with this same property (this
  $\tilde{g}$ may not fix $x$ and may not belong to $G(\O)$). Thus we have that $gF_{1}g^{-1} \in G(\cZ) \subset W^{8}$ and therefore $F_{1} \in W^{12}$. 
  
  Therefore, we have $F = (f_\beta^{-1}f_\alpha)h(f_\beta^{-1}f_\alpha)^{-1} \in W^{20}$ and so $h \in W^{24}$. 
  \end{proof}
  
  \subsection{Step 5: Wrapping Up}
  
  We will first verify a version of the Steinhaus condition for stable neighborhoods. 

\begin{thm}\label{thm:globalpointed}
  Let $X$ be a second countable Stone space, $\O$ a stable
  neighborhood of $x\in X$, and $G=\Homeo(X)$. Suppose $W$ is a
  symmetric set in $G$ such that:
  \begin{itemize}
    \item $G$ is covered by countably many sets $g_iW$, $g_i\in G$,
      and
    \item $W^2$ is dense in $G(\O)$.
  \end{itemize}
  Then $G(\O)\subseteq W^{96}$.
\end{thm}

\begin{proof}
    Let $h \in G$. We first apply \Cref{lem:fpfragment} and \Cref{lem:brickfragment} in order to write $h=h_{1}h_{2}h_{3}h_{4}$ where each $h_{i} \in G(\O)$ fixes a point $x_{i} \in G(\O) \cdot x$ so that $\O$ is telescoping with respect to $x_{i}$ and the support of $h_{i}$ is contained in a brick around $x_{i}$. Next we apply \Cref{lem:pigeonhole} to see that each $h_{i}$ is contained in $W^{24}$. We conclude that $h \in W^{96}$.
  \end{proof}
  
  Finally, we can upgrade this to the case that $X$ is stable. We first need a short lemma on stable spaces. 
  
  \begin{lemma}\label{lem:stablepartition}
  Suppose $X$ is stable. Then $X$ can be written as a
  finite disjoint union of spaces each of which is a stable neighborhood of one of its points.
\end{lemma}

\begin{proof}
  Using compactness, cover $X$ by finitely many clopen sets
  $T_1,\cdots,T_N$ so that each $T_i$ is a stable neighborhood of a
  point $x_i\in T_i$, with all $x_i$ distinct from each other. The
  sets $T_i$ may not be disjoint. We perform the following
  operation to make them disjoint. Suppose $T_1\cap
  T_2\neq\emptyset$. If $x_1\not\in T_2$, then replace $T_1$ with
  $T_1\smallsetminus T_2$; this space is also stable with respect
  to $x_1$, as stability passes to sub-neighborhoods. Similarly, if $x_2\not\in
  T_1$ we can replace $T_2$ by $T_2\smallsetminus T_1$. Now suppose
  $x_1,x_2$ are both in $T_1\cap T_2$. Write $T_1\cap T_2=A_1\sqcup
  A_2$ where $A_i$ are clopen and $x_i\in A_i$. Then replace $T_1$
  with $T_1\smallsetminus A_2$ and $T_2$ with $T_2\smallsetminus
  A_1$. Again both are stable neighborhoods of the same
  points. Thus in all cases we replaced $T_1,T_2$ with two disjoint
  stable neighborhoods of the same points and with the same
  union as $T_1$ and $T_2$. Continuing in this way produces the
  desired partition.
\end{proof}

We will need a slight extension of Lemmas \ref{lem:diagonal},
\ref{lem:diagonal2}, \ref{lem:pigeonhole} and Theorem
\ref{thm:globalpointed}.

Suppose we have a finite collection of distinct points $x_i\in X$,
$i=1,\cdots,N$ and pairwise disjoint clopen subsets $\Omega_i$, $i=1,\cdots,N$
so that $\Omega_i$ is a stable neighborhood of $x_i$ for all $i$. We will
abbreviate this by $\vec \O$ and $\vec x\in\vec \O$. A {\it multibrick} $\vec\cB$
is a collection of bricks $\cB_i$ in each $\O_i$ centered at $x_i$. The group $\Homeo(\vec
\O)$ is the group of homeomorphisms of $\sqcup \O_i$ preserving each
$\O_i$, and $\Homeo(\vec \O,\vec x)$ is the subgroup that in addition
preserves each $x_i$.

\begin{prop}\label{omnibus}
  Lemmas \ref{lem:diagonal},
\ref{lem:diagonal2}, \ref{lem:pigeonhole} and Theorem
\ref{thm:globalpointed} hold when $G(\O)$, $\Homeo(\O,x)$ and bricks
are replaced with $\Homeo(\vec \O)$, $\Homeo(\vec \O,\vec x)$ and
multibricks.
\end{prop}

\begin{proof}
  Proofs remain valid, {\it mutatis mutandis}.
\end{proof}

We are finally ready to prove the main theorem of this section.

  \begin{proof}[Proof of \Cref{thm:endsmain}]
    We verify the Steinhaus property. Let $W$ be a symmetric set in
    $G=\Homeo(X)$ so that $G$ is the union of countably many translates
    of $W$. By \Cref{lem:baire}, there is an open
    neighborhood $\cU$ of the identity in $G$ so that $W^2$ is dense in
    it. There is a finite clopen partition $A_1\sqcup\cdots\sqcup A_m$ of $X$
    so that any $g\in G$ that preserves each $A_i$ belongs to
    $\cU$. Now applying \Cref{lem:stablepartition} to each $A_i$
    produces a finer clopen partition $\O_1\sqcup\cdots\sqcup \O_N$ of $X$ where each partition element is
    a stable neighborhood. This finer partition determines an open set
    $\cU' \subset \cU$. Every $g \in \cU'$ thus restricts on each
    subset in the partition to a map on a stable neighborhood. We can
    apply the Proposition \ref{omnibus} version of \Cref{thm:globalpointed} to see that $g \in W^{96}$. 
  \end{proof}

\section{Surfaces}\label{sec:surfaces}

The goal of this section is to prove the classification theorem for stable surfaces as stated in the introduction (recalled here).

\begin{customthm}{A} 
	Let $\S$ be a stable, orientable surface without boundary. The homeomorphism group and mapping class group of $\S$ have automatic continuity if and only if every end of $\S$ is telescoping. That is, every end is one of the following:
	\begin{enumerate}
		\item An isolated puncture,
		\item of Cantor type, or
		\item is not isolated in the space of ends accumulated by genus and is a successor with all predecessors of Cantor type. 
	\end{enumerate}
\end{customthm}

\subsection{Background}\label{ssec:surfacebackground}
We recall some terminology and structural lemmas on end spaces of surfaces introduced in \cite{MR2023}. For a surface $\S$ 
we denote by $E(\S)$ the space of ends of $\S$. This space is defined as 
\begin{align*}
	E(\S) = \varprojlim_{K} \pi_{0}(\S \setminus K),
\end{align*}
where the inverse limit is taken over all connected, compact subsurfaces of $\S$. We denote by 
\[
E_{g}(\S) \subset E(\S)
\] 
the subset of $E(\S)$ consisting on non-planar ends. 
The end space is equipped with the inverse limit topology. This makes it into a second countable Stone space and $E_{g}(\S)$ into a closed subset. 
When the underlying surface is fixed we will often shorten the pair $(E(\S),E_{g}(\S))$ to just $E(\S)$.

Every homeomorphism of $\Sigma$ induces a homeomorphism of $E(\S)$ that preserves $E_{g}(\S)$. 
Hence, when we talk about a homeomorphisms between subsets of $x$, we always assume that our maps 
preserve the set $E_{g}(\S)$. That is, we think points in $E_{g}(\S)$ as being colored following the 
convention of Section~\ref{sec:endspaces}. Furthermore, this pair $(E(\S),E_{g}(\S))$ effectively classifies infinite-type surfaces.

\begin{THM}\label{thm:classification}\cite{Kerekjarto1923,Richards1963}
	Let $\S$ be an orientable surface with finitely many boundary components. Then $\S$ is determined up to homeomorphism by the triple $(g,b,(E(\S),E_{g}(\S)))$ where $g \in \Z_{\geq 0 } \cup \{\infty\}$ is the genus of $\S$, $b \in \Z_{\geq 0}$ is the number of boundary components of $\S$, and the pair $(E(\S),E_{g}(\S))$ is considered up to homeomorphism.  Furthermore, the quotient map $\Homeo(\S) \rightarrow \Homeo(E(\S),E_{g}(\S))$ is a continuous surjection. 
\end{THM}

Given a clopen subset $U \subseteq (E(\S),E_{g}(\S))$ we will let
$\S_{U}$ denote a connected subsurface of $\S$, closed in $\S$, with either infinite or zero genus so that 
$\partial \S_{U}$ is a single simple closed curve and the space of ends of $\S_U$, $E(\S_{U})$, is  $U$. Given a neighborhood $U$ of an end $x \in E$, we will say that $\S_{U}$ is a neighborhood of $x$ 
in the surface $\S$. Furthermore, if we have a pair $U \subset V$ we will assume that $\Sigma_{U}$ and $\Sigma_{V}$ are chosen so that $\Sigma_{U} \subset \Sigma_{V}$. Note that $\Sigma_U$ has infinite genus if and only if $U \cap E_{g}(\S) \not = \emptyset$.  
\begin{LEM} 
Given subsurfaces $\S_U$ and $\S_V$, we have $\S_U$ is homeomorphic to $\S_V$ if an only if $U$ is homeomorphic to $V$.
\end{LEM}
\begin{proof}
A homeomorphism between $\S_U$ and $\S_V$ induces a homeomorphism between their end spaces that 
preserves the sets of planar and non-planar ends. Hence, if $\S_U$ if homeomorphic to $\S_V$ then $U$ is
homeomorphic to $V$. 

In the other direction, as in Section~\ref{sec:endspaces}, any homeomorphism $\phi : U \to V$ is assumed to send the set 
$U \cap E_{g}(\S)$ to $V \cap E_{g}(\S)$.  In particular, $\S_U$ has infinite genus (or genus zero) if and only if $\S_V$ does. 
Since they both have one boundary component, same genus and homeomorphic end spaces, they are 
homeomorphic by the classification of surfaces. 
\end{proof} 

\begin{DEF}\cite[Definition 4.14]{MR2023} \label{stable-surface}
	We say a surface $\S$ is \textbf{stable} if $E(\S)$ is stable as in Definition~\ref{stable}. 
	If $U$ is a stable neighborhood of $x$ in $E(\S)$, then $\S_{U}$ is a \textbf{stable} neighborhood of $x$ in $\S$. 
\end{DEF}

\begin{DEF}\cite[Definition 4.1]{MR2023}
	Let $\preceq$ be the preorder on $E(\S)$ defined by $y \preceq x$ if for every neighborhood $U$ of $x$, there exists a neighborhood $V$ of $y$ and a homeomorphism $f \in \Homeo(S)$ so that $f(V) \subset U$. Write $x \sim y$ if $y \preceq x$ and $x \preceq y$. Recall that, by \cite[Theorem 1.2]{MR2024}, if $x \sim y$ then there is 
	a homeomorphism $\phi :\S \to \S$ such that $\phi(x) = y$. Denote by $E(x)$ the equivalence class of $x$. 
\end{DEF}

\begin{DEF} \label{def:successor} 
       We say $y \prec x$ is $y \preceq x$ and $y \not \sim x$. If $y  \not \preceq x$ and $x \not \preceq y$
       we say $x$ and $y$ are \textbf{incomparable}. 
	An end, $x$, is a \textbf{successor} if there exists finitely many 
	incomparable ends $y_{1},\ldots,y_{n}$, with $n\neq0,$ so that each 
	$y_{i} \prec x$ and if $z \prec x$, then $z \preceq y_{i}$ for some $i$. 
	Each $y_{i}$ is a \textbf{predecessor} of $x$. 
	\end{DEF}

Abusing notation, we denote the partial order on the set of
equivalence classes induced by $\preceq$ again by $\preceq$. 

\begin{PROP}\cite[Proposition 4.7]{MR2023}
	The partial order $\preceq$ has maximal elements. Furthermore, for any maximal element $x$, $E(x)$ is either a finite set of points or a Cantor set. 
\end{PROP}

We say that an end $y$ is of \textbf{Cantor type} if there exists some neighborhood $U$ of $y$ such that $E(y) \cap U$ is homeomorphic to a Cantor set.
Note that $E(y)$ may not be a Cantor set if $y$ is not maximal. For example, it is possible that only a finite number of points in 
$\overline E(y)$ are non-planar. However:

\begin{LEM} \label{lem:Cantor-type}
If $\S$ is stable, then, for every end $y$ of $\S$, the set $E(y)$ is locally closed. 
In particular, an end $y$ is of Cantor type if and only if $E(y)$ is uncountable. 
\end{LEM} 

\begin{proof} 
Let $V$ be a stable neighborhood of $y$. 
For every $z \in V$, $E(z)$ intersects every stable neighborhood of $y$,
and hence, $z \preceq y$. This implies that either $z \sim y$ and $z \in E(y)$ 
or $y \not \preceq z$ and $z \not \in \overline{E(y)}$. That is
\[
\overline{E(y)} \cap V = E(y) \cap V
\]
and hence $E(y) \cap V$ is closed. 

To see the second assertion, it is immediate from the definition that if $y$ is of Cantor type, then $E(y)$ is uncountable. In the other direction, since  $E(y)$  is uncountable and  $E(\S)$  is compact, there exists at least 
one point in  $E(y)$  that is an accumulation 
point of  $E(y)$. Moreover, because all stable neighborhoods of points in  $E(y)$ are homeomorphic to one another, 
every point in  $E(y) \cap V$  is an accumulation point of  $E(y)$. This implies that  $E(y) \cap V$  is a perfect set. 
Given that  $E(\S)$  is totally disconnected, $ E(y) \cap V$  forms a Cantor set, which means that  $y$  is of Cantor type.
\end{proof} 

Note that, for a maximal end $x$, saying $x$ is of Cantor type is the same as saying $E(x)$ is a Cantor set. 
This is because, for every point $y \in \overline{E(x)}$, we have $x \preceq y$ and, since $x$ is maximal, $x \preceq y$. 

\begin{PROP}\label{prop:stablefinmax}
	If $\S$ is a stable surface, then $\S$ has only finitely many maximal equivalence classes of ends. 
\end{PROP}

This proposition will follow from the following. The version we have written here is not exactly the same as was written in \cite{MR2023}, but follows directly from their statements.

\begin{LEM}\cite[Proposition 4.8 \& Remark 4.15]{MR2023}\label{lem:MRstable}
	If $U$ is a stable neighborhood of $x$, then $U$ has only a single maximal equivalence class of ends. Furthermore, $x$ is a maximal end of $U$ and is either the unique maximal end of $U$ or is of Cantor type. 
\end{LEM}

\begin{proof}[Proof of \Cref{prop:stablefinmax}]
	Let $\S$ be a stable surface. Using compactness of the end space of $\S$ we can cover $E(\S)$ by a finite collection of stable neighborhoods $U_{1},\ldots,U_{n}$. Now, by the previous lemma, each $U_{i}$ has only a single type of maximal end. Furthermore, each maximal end in $E(\S)$ must appear as a maximal end of some $U_{i}$. Therefore we see that $n$ is an upper bound on the number of maximal equivalence classes of ends.
\end{proof}

Next, we recall the following.

\begin{LEM}\cite[Lemma 4.18]{MR2023} \label{lem:Uy}
Let $x, y \in  E(\S),$ and assume $x$ has a stable neighborhood $V_x$ and that $x$ is an accumulation point of $E(y)$. Then 
for any sufficiently small clopen neighborhood $U_y$ of $y$, $U_y \cup V_x$ is homeomorphic to $V_x$.
\end{LEM}

We restate this lemma slightly. 

\begin{COR} \label{cor:accumulate}
Let $\S$ be a stable surface and $V \subset E(\S)$ be a clopen set such that, for every $y \in (E\setminus V)$, $E(y)$ has an accumulation 
point in $V$. Then $V$ is homeomorphic to $E(\S)$. 
\end{COR}

\begin{proof} 
For every $y \in E\setminus V$, let $x \in V$ be an accumulation point of $E(y)$. 
Let $V_x \subset V$ and $U_y \subset E\setminus V$ be as in Lemma~\ref{lem:Uy}. 
Then 
\[
V \cup U_y = (V\setminus V_x) \cup (V_x \cup U_y) \simeq (V\setminus V_x) \cup V_x \simeq V. 
\]
Since $E\setminus V$ is compact, it can be covered with finitely many such sets 
$U_{y_1}, \dots, U_{y_k}$. Making these sets smaller, we can assume they are disjoint. Then,
from the above argument, we have 
\[
E = (V \sqcup U_{y_1}) \sqcup U_{y_2} \sqcup \dots \sqcup U_{y_k} \simeq
(V \sqcup U_{y_2}) \sqcup  U_{y_3} \sqcup \dots \sqcup U_{y_k} \simeq
\dots \simeq (V \sqcup U_{y_k} ) \simeq V
\]
This finishes the proof.
\end{proof} 

\subsection{Telescoping and Big Annuli}

From now on, we will always be working under the assumption that $\S$ is a stable surface.

\begin{definition}\label{def:telescopingend}
	We say that an end $x$ of $\S$ is \textbf{telescoping} if it is one of the following:
	\begin{enumerate}[(i)]
		\item An isolated puncture,
		\item of Cantor type, or
		\item is not isolated in $E_{g}(\S)$ and is a successor with all predecessors of Cantor type. 
	\end{enumerate}
\end{definition}

\begin{definition}\label{def:bigannuli}
	For an end $x$ of $\S$, an \textbf{annulus around} $x$ is a subsurface $A \subset \S$ with $\partial A = \alpha_{1} \sqcup \alpha_{2}$ 
	such that each $\alpha_{i}$ is a single simple closed curve, $\alpha_1$ is the boundary of $\S_{U_1}$ and $\alpha_2$ is the boundary of $\S_{U_2}$
	where $U_2 \subset U_1$ are nested stable neighborhoods of $x$. 
	
	Let $U$ be a stable neighborhood of $x$, let $A$ be an annulus in the interior of $\S_U$ around $x$, and let $A'$ be the component of $\S_U-A$
	that does not contain $x$. We say $A$ is a \textbf{big annulus} for the neighborhood $U$ of $x$ if both $A$ and $A'$ contain all types of ends in 
	$U$ other than possibly $x$ itself. If $x$ is an isolated puncture, then $\S_U$ is a punctured disk and every annulus in $\S_U$ around $x$ is a big annulus. 
\end{definition}

Now we check some basic properties of stable neighborhoods and big
annuli relative to a telescoping end, as well as prove an alternate characterization of telescoping.

\begin{LEM} \label{lem:bigannulusproperties}
Let $x$ be a telescoping end of $\S$ and let $U$ be a stable neighborhood of $x$. 
Then any two big annuli in $\S_U$ around $x$ are homeomorphic via a homeomorphism of $\S_U$ that fixes $x$.
\end{LEM}

\begin{proof} 
Let $A$ and $B$ be two big annuli in $\S_U$. Let $A'$ be the component of $\S_U-A$
that does not contain $x$ and $A''$ be the component that does contain $x$. 
Let $B'$ and $B''$ be similarly defined. 

Let $V \subset U$ be a small neighborhood of $x$ that is disjoint from $E(A)$ and $E(B)$. 
The end space $E(A')$ is a subset of the end space of $E(U-V)$. Every 
$y \in E(U-V)$ is either a predecessor of $x$ or accumulates to a predecessor of $x$. 
Either way, $E(y)$ has an accumulation point in $E(A')$. Corollary \ref{cor:accumulate}
implies that $E(A')$ is homeomorphic to $E(U-V)$. Also, since $x$ is an stable end, 
$\S_V$ is homeomorphic to $A \cup A''$. Therefore, there is a homeomorphism $\phi_A$ of 
$\S_U$ fixing $x$ that sends $A \cup A''$ to $\S_V$ and $A'$ to $\S_{U}-\S_V$. 
Similarly, there is a homeomorphism $\phi_B$ of $\S_U$ fixing $x$ that sends $B \cup B''$ to $\S_V$ and $B'$ 
to $\S_{U}-\S_V$. Then $\phi_B^{-1}\phi_A$ sends $A'$ to $B'$. 

Hence, we can assume $A' = B'$. Now applying the same argument, we can find a homeomorphism 
$\psi_A$ of $A \cup A''$ fixing $x$ that sends $A''$ to $\S_V$ and a homeomorphism 
$\psi_B$ of $B \cup B''$ fixing $x$ that sends $B''$ to $\S_V$. Then $\phi_B^{-1}\phi_A$ sends $A$ to $B$
and we are done.  
\end{proof} 

\begin{LEM}\label{lem:telnbhdannuli}
Let $x$ be an end and $U$ be a stable neighborhood of $x$. Then 
$x$ is a telescoping end of $\S$ if and only if there is exists a subdivision 
\[
\S_U = \bigcup_{i=0}^{\infty} Y_{i}
\] 
of $\S_U$ into annuli $Y_i$ around $x$ so that
\begin{enumerate}[(i)]
	\item $Y_{i} \cap Y_{j} = \emptyset$ if $|i-j| >1$, 
	\item $Y_{i}$ and $Y_{i+1}$ intersect only along one of their boundary components, and 
        \item for $0 \leq i \leq j$, we have $\cup_{k=i}^j Y_k$ is homeomorphic to $Y_0$. 
	\end{enumerate}
In particular, each $E(Y_i)$ contains all types of ends in $U$ other than possibly $x$ itself, and for $i \geq 1$, $Y_i$ is a big annulus. 	
\end{LEM}

\begin{proof}
Assume $x$ is a telescoping end and let $U$ be a stable neighborhood of $x$. Let 
\[
U=U_0 \supset U_1 \supset U_2 \supset\dots
\] 
be a family of nested subsets of $U$ that descends to $x$, $U_k \searrow x$. We immediately note that if $x$ is an isolated puncture, then we can take the collection $\{Y_{i}\}$ to be nested annuli that descend to $x$. From now on we assume that $x$ is not an isolated puncture.
First we show that, for $k_1$ large enough, $U - U_{k_1}$ contains all types of ends in $U$ other 
than possibly $x$ itself. 

There are two cases to consider. If $x$ is of Cantor type then, for every end $y \in U$, $E(y) \cap U$
accumulates to $x$ and hence its accumulates to every point in $E(x) \cap U$. So, we only need to 
take $k_1$ large enough so that $(U - U_{k_1}) \cap E(x)$ is not empty and this ensures that 
$(U - U_{k_1})$ contains all types of ends.
If $x$ is not of Cantor type then it has finitely many predecessors $y_1,...,y_n$ and they are all 
of Cantor type. Then for every end $z \in U$, $E(z)$ accumulates to some $y_i$ and hence to 
all points in $E(y_i) \cap U$. Hence, if $k_1$ large enough so that $(U - U_k) \cap E(y_i)$ 
is not empty for $i=1, \dots, n$ then $(U - U_k)$ contains all ends except possibly $x$ itself. 

We proceed in this way and choose indices $k_1, k_2, \dots$ such that $U_{k_i} - U_{k_{i+1}}$ contains 
all types of ends in $U$ other than possibly $x$ itself. Let $\S_U$ be a stable neighborhood of 
$x$ in $\S$ and denote the boundary of $\S_U$ by $\alpha_0$.  Choose a sequence of curves $\alpha_i$ 
in $\S_U$ that exit the end $x$ such that the component of $\S_U - \alpha_i$ that contains $x$ has the ends 
space $U_{k_i}$. Let $Y_i $ be the annulus bounded by $\alpha_i$ and $\alpha_{i+1}$. 
Then $\S_U = \bigcup_{i=0}^{\infty} Y_{i}$ and the first two assumptions of the lemma hold by construction. 

To see that the final assumption holds, it is sufficient to show that $Y_0 \cup Y_1$ is homeomorphic to
both $Y_0$ and $Y_1$. If $\S_U$ has infinite genus, then $x$ is a non-planar end. Since 
$x$ is not isolated in $E_G$, some ends in some $Y_i$ (and hence every $Y_i$) have to be non-planar. 
Therefore, every $Y_i$ has infinite genus. Otherwise, $\S_U$ has genus zero and hence
every $Y_i$ has genus zero. Hence, $Y_0 \cup Y_1$, $Y_0$ and $Y_1$ have the same genus. 
Also, they all have two boundary components. We only need to check that their end spaces are
homeomorphic. But this follows from Corollary~\ref{cor:accumulate} applied to $E(Y_0) \subset E(Y_0 \cup Y_1)$
or to $E(Y_1) \subset E(Y_0 \cup Y_1)$. This finishes the proof of the one direction. 

To see the other direction, let $\S_U$ be a stable neighborhood of $x$ in $\S$ and let 
$\S_U = \bigcup_{i=0}^{\infty} Y_{i}$ be a decomposition of $\S_U$ into annuli as described in the statement. 
We show tha $x$ is a telescoping end by checking the conditions of Definition~\ref{def:telescopingend}. 

If $x$ is the only non-planar end in $\S_U$, then there is an annulus $Y_k$ that has a finite 
non-zero genus. The $Y_0$ also has finite non-zero genus and hence $Y_0$ cannot be homeomorphic to 
$Y_0 \cup \dots \cup Y_k$ which contradicts assumption (iii). Therefore, either $x$ is a planar end, 
or it is not isolated in $E_G(\S)$.

If $x$ is an isolated puncture or is of Cantor type, then we are done by definition. Otherwise, we need to show
that $x$ is a successor and the predecessors are of Cantor type. By Lemma~\ref{prop:stablefinmax}, 
$E(Y_0)$ has finitely many different maximal types,
say $y_1, \dots, y_n$. Then every $z \in U$, $E(z)\cap E(Y_0)$ is non-empty and hence 
$E(z)$ accumulates to some $y_i$. This, by Definition~\ref{def:successor}, implies that 
$x$ is a successor. 

The assumption (iii) in particular implies that $E(Y_0)$ is homeomorphic to $E(Y_k)$ for every $k$
and in particular $E(y_i) \cap E(Y_k)$ is a non-zero for every $k$. If f $E(y_i) \cap E(Y_0)$ is finite, 
then $E(y_0) \cap (Y_0 \cup Y_1)$ is larger than $E(y_0) \cap Y_0$
and hence $Y_0$ cannot be homeomorphic to $Y_0 \cup Y_1$ which contradicts assumption (iii). 
If $E(y_i) \cap E(Y_0)$ is infinite then $E(y_i)$ has an accumulation point in $E(Y_0)$.
But $y_i$ is maximal in $E(Y_0)$ hence the accumulation point has to be in $E(y_i)$. 
Therefore every point in $E(y_i) \cap E(Y_0)$ is an accumulation point of $E(y_i)$ and
$E(y_i)$ is uncountable. Lemma~\ref{lem:Cantor-type} implies that $y_i$ is of Cantor type.
This holds for every  $i=1, \dots, n$, therefore $x$ is a telescoping end.
\end{proof}

We have now also seen that
the definition of telescoping, \Cref{def:telescopingend}, fits into
the same framework as \Cref{prop:stable} and \Cref{stable} from the end space case. Here a
\textbf{brick} is a subsurface $\cB$ of $\S_U$ which is a disjoint union of big
annuli so that all complementary components contain all types of ends
in $U$ except possibly $x$. For example, fixing a decomposition
$\S_U=\cup_{i=0} Y_i$ as in \Cref{lem:telnbhdannuli}, the union $\cup
Y_{n_i}$ is a brick for any infinite and co-infinite subset
$\{n_i\}\subset\{1,2,\cdots\}$ (it is important not to include $Y_0$
since $Y_0$ is not a big annulus). From
\Cref{lem:bigannulusproperties} we see that for any two bricks there
is a homeomorphism of $\S_U$ taking one to the other. In particular,
any brick can be enlarged by inserting big annuli into the
complementary components, including the complementary component that
contains $\partial\S_U$, and therefore we could choose a homeomorphism
between two bricks to fix a neighborhood of $\partial\S_U$ that
contains a big annulus.

\begin{assumptions} As in the end space case we now set up some notation and standing assumptions. Unless explicitly stated, these will hold until the end of \Cref{ssec:step5surface}.
\begin{enumerate}[(1)]
\item If $A$ is an annulus or a finite type subsurface of $\S$ we will
  use $G(A)$ to denote the subgroup of $G=\Homeo(\S)$ consisting
  of maps supported on the interior of $A$. If $U$ is a stable
  neighborhood of a telescoping end $x$, then $G(\S_U)$ is the set of
  homeomorphisms supported on some $\S_{U'}$ so that $U\smallsetminus
  U'$ contains all types of ends in $U$ except possibly $x$. 

\item Fix a telescoping end $x$ and $\S_U =
\bigcup_{i=0}^{\infty} Y_{i}$ a subdivision as in \Cref{lem:telnbhdannuli}. 

\item Assume that $W\subset G$ is a symmetric set so that $G$ is covered by countably many sets $g_{i}W$ with $g_{i}\in G$ and $W^{2}$ is dense in $G(\S_{U})$ (i.e. $\overline{W^{2}} \supset G(\S_{U})$).
  
  \item For any brick, $\cB$, we write $G(\cB)$ for the closed subgroup of homeomorphisms of $\S_{U}$ or $\S$ supported on $\cB$. 
\end{enumerate}
\end{assumptions}

\subsection{Step 1: Fragmentation}

\begin{lemma}\label{lem:fpfragmentS}
	Any $h \in G(\S_U)$ can be written as $h=h_{1}h_{2}$ so that
        each $h_{i}$ belongs to $G(\S_U)$ and fixes some point in $G(\S_U) \cdot x$. 
\end{lemma}

\begin{proof}
  Since $h\in G(\S_U)$, by definition there is a smaller stable
  neighborhood $\S_{U'}$ so that $h$ is supported in $\S_{U'}$ and
  $U\smallsetminus U'$ contains all types of ends as $U$ except
  possibly $x$. We can modify the subdivision $Y_i$ so that
  $\S_{U'}=\cup_{i=1}^\infty Y_i$ (e.g. choose a homeomorphism
  $\S_U\to\S_{U'}$ fixing $x$ and let the new $Y_i$ be the image of
  the (old) $Y_{i-1}$ and set $Y_0 = U \setminus U'$).
	Suppose $h(x)\neq x$ and assume
    $h^{-1}(x)\in Y_i$ for some $i>0$. Choose some $j\neq i$, $j>0$ and note
    that $h(Y_j)$ is disjoint from a neighborhood of $x$, so it is
    contained in some finite union
    $Y_{k_1}\cup\cdots\cup Y_{k_r}$. Using the classification of surfaces we can find a
    homeomorphism $h_1\in G(\S_U)$ supported on only finitely many of
    the annuli including $Y_0$ and
    agreeing with $h$ on $Y_j$. Thus $h=h_1h_2$ where $h_1$ fixes
    $x$ and $h_2=h_1^{-1}h$ is identity on $Y_0$ and $Y_j$. Finally we note that since $Y_{i}$ and $Y_{j}$ are both big annuli, they are homeomorphic, and so $Y_{j}$ contains an end in $G(\S_U) \cdot x$ and $h_{2}$ fixes this end. 
\end{proof}

Again, note that if $U$ is a stable neighborhood of an end $x$, then $U$ is also a stable neighborhood of $g \cdot x$ for all $g \in G(\S_U)$. Keeping this in mind, the statements that follow will apply when one replaces the telescoping end $x$ with an end of the form $g \cdot x$ for $g \in G(\S_{U})$. We will slightly abuse notation and not keep track of this distinction.

\begin{lemma}\label{lem:brickfragmentS}
  Any $h\in G(\S_{U})$ that fixes $x$ can be written as $h=h_1h_2$ so that each $h_i$ is
  supported on a brick.
\end{lemma}

\begin{proof}
	Given $h$, choose two sequences $(n_{i})$ and $(m_{i})$ so that $1=n_1<m_1<n_2<m_2<\cdots$, $h(Y_1\cup Y_2\cup\cdots\cup Y_{n_{i}})$ is disjoint from $Y_{m_{i}}$ and $h(Y_{n_{i+1}})$ is disjoint from $Y_{m_{i}}$. Then define $h_1$ to agree with $h$ on each $Y_{n_i}$ and to be the identity on $Y_{m_{i}}$ for all $i$. This can be done by the classification of surfaces. Then $h_1$ and $h_2=h_1^{-1}h$ are both supported on bricks of $\S_U$.
\end{proof}

Just as in the end space case, these two lemmas allow us to write any
$h \in G(\S_U)$ as $h=h_{1}h_{2}h_{3}h_{4}$ with each $h_{i}\in
G(\S_U)$ supported on a brick limiting to a fixed point in the orbit
of $x$. 

\subsection{Step 2: Finding Commutators}\label{ssec:step2S}

For surfaces this step will be drastically different than in the end space case. For end spaces, since we were dealing with totally disconnected spaces, we could use maps that shift an entire brick and restrict to the identity elsewhere. Now the topology of the surface obstructs using such a shift. Instead, we will fragment our map again using an Eilenberg-Mazur swindle \cite{Bass1963,Mazur1959} into two maps, still supported on bricks, that are themselves commutators. 

\begin{DEF}
	Let $A$ be a big annulus in $\S_U$. We say that $f \in G(A)$ has
        \textbf{alternating support} if $\supp(f) = A_{1} \sqcup
        A_{2}$, with each $A_{i}$ a big annulus and with all 3
        complementary components of $A_1\sqcup A_2$ in $A$ big annuli,
        and so that there exists $h \in G(A)$ so that $h(A_{1}) =
        A_{2}$ and $f\vert_{A_{2}} = h(f\vert_{A_{1}})^{-1}
        h^{-1}$. If $\cA$ is a brick, then we say that $f \in G(\cA)$ has
        \textbf{alternating support} if $f\vert_{A}$ has alternating
        support for each annulus component $A \in \cA$. 
\end{DEF}

\begin{LEM}\label{lem:altsuppcomm}
	If $f \in G(A)$ has alternating support, then $f$ can be written as a single commutator of maps in $G(A)$. Similarly, if $f \in G(\cA)$, for $\cA$ a brick, has alternating support, then $f$ can be written as a single commutator of maps in $G(\cA)$. 
\end{LEM}

\begin{proof}
	Let $h$ be as in the definition of alternating support. Let $f_{1}$ denote the homeomorphism that agrees with $f$ on $A_{1}$ and is the identity outside of $A_{1}$. Then we can write 
	\begin{align*}
		f = f_{1} h (f_{1})^{-1} h^{-1} = [f_{1},h].
	\end{align*} 
	If $\cA$ is a brick then we can simply apply this to each annulus of $\cA$ and collect terms since the annuli are all pairwise disjoint
\end{proof}

\begin{LEM}(Eilenberg-Mazur Swindle) \label{lem:EMfragment}
	 Let $f \in G(\cA)$ for $\cA$ be a brick of $\S_U$. There exist two
         bricks, $\cA_{1}$ and $\cA_{2}$ in $\S_{U}$, and $h_{i} \in G(\cA_{i})$
         so that $h_{i}$ has alternating support for $i=1,2$ and $f=h_{1}h_{2}$.
\end{LEM}

\begin{proof}
  The brick $\cA$ is the disjoint union of big annuli $A_j$ ordered linearly,
  with $A_{j+1}$ separating $A_j$ from $x$. Denote by $f_j=f|A_j$ the
  restriction of $f$ to $A_j$. So we can encode $f$ as
  $(f_1,f_2,\cdots)$. We will economize on the notation and write this
  as
  $$\red{1\quad 2 \quad 3\quad 4 \quad\cdots}$$
  Now consider the map defined by
  $$\red{1}\blue{\overline 1 1}\red{2}\blue{\overline 1 \overline 2 1
    2}\red{3} \blue{\overline 1 \overline 2 \overline 3 123}\red{4}
  \blue{\overline 1 \overline 2 \overline 3 \overline 4\cdots}$$ where each of the numbers represents a homeomorphism supported on a
        big annulus,
        and we always assume that between any two
        consecutive such big annuli there is a big annulus (which can
        then be subdivided into any finite number of big annuli). E.g., each occurence of $\overline 1$ denotes the homeomorphism that does exactly $f_1^{-1}$ on a single big annulus that is separated by big annuli from the supports of each other map. Thus we are inserting two big annuli (with separation)
  between $A_1$ and $A_2$ and we put on them copies of $f_1^{-1}$ and
  $f_1$ respectively, inserting 4 big annuli between $A_2$ and $A_3$
  etc. Call this map $h_1$ and let $h_2=h_1f^{-1}$. Thus $h_2$ is
  obtained from the above sequence by restricting to the blue
  numbers. Each block of blue numbers between two red numbers defines
  a map on a single big annulus with alternating supports:
  $$\red{1}\overbrace{\blue{\overline 1}\blue{1}}\red{2}\overbrace{\blue{\overline 1 \overline 2}\blue {1
    2}}\red{3} \overbrace{\blue{\overline 1 \overline 2 \overline 3}\blue{123}}\red{4}\overbrace{\blue{\overline 1 \overline 2 \overline 3 \overline 4}\blue{1234}}\cdots$$
   Therefore $h_2$ has alternating supports. 
  
  But so does $h_1$:
$$\underbrace{\red{1}\blue{\overline 1}}\underbrace{\blue{1}\red{2}\blue{\overline 1 \overline 2}}\underbrace{\blue {1
    2}\red{3} \blue{\overline 1 \overline 2 \overline 3}} \underbrace{\blue{123}\red{4}
  \blue{\overline 1 \overline 2 \overline 3 \overline 4}}\cdots$$
   Here we get alternating supports corresponding to the groups of
        numbers indicated after adding a big annulus on each side of
        the group (and using the definition of a brick that guarantees
        that there is a big annulus to the left of the
        first group).
\end{proof}

The culmination of this step is to combine these two lemmas in order
to write a map supported on a brick as a product of two commutators,
each of which are also supported on the brick.

\begin{lemma}\label{lem:fragcommutators}
	Any $h \in G(\S_U)$ supported on a brick can be written as
        $h=h_{1}h_{2}$ so that each $h_{i}\in G(\S_{U})$ is supported
        on a brick $\cA_{i}$ in $\S_{U}$ and can be written as a single commutator in $G(\cA_{i})$. 
\end{lemma}

\subsection{Step 3: Diagonalization}\label{ssec:step3S}

Again, we next want to find a ``good'' brick. The first lemma is identical to the end space case. 

\begin{lemma}\label{lem:diagonalS}
  Let $\cA$ be a brick and $\cA_1,\cA_2,\cdots$ pairwise disjoint
  subbricks. Then 
  \begin{enumerate}[(i)]
    \item there is some $i\geq 1$ such that any homeomorphism of $\cA_i$
      extends to a homeomorphism of $\S$ supported on $\cA$ that belongs to $g_iW$, and
      \item moreover, every homeomorphism of $\cA_i$
        extends to a homeomorphism of $\S$ supported on $\cA$ that belongs to $W^2$.
  \end{enumerate}
\end{lemma}

\begin{proof}
  This is exactly the same proof as the proof of \Cref{lem:diagonal}.
\end{proof}

Next we upgrade this approximation. Notably, here we will only be doing it for maps with alternating support. 

\begin{lemma}\label{lem:diagonal2S}
	Let $\cA$ be a brick. Then there exists a subbrick $\cZ$ of $\cA$ so that any homeomorphism with alternating support on $\cZ$ is in $W^{8}$. 
\end{lemma}

\begin{proof}
	Apply the previous lemma to $\cA$ (and an arbitrary pairwise
        disjoint collection of subbricks) and label the resulting
        subbrick $\cB$. Repeat this process to $\cB$ and label the
        subbrick $\cZ$. Let $f \in G(\cZ)$ have alternating
        support. Then by \Cref{lem:altsuppcomm}, $f = [u,v]$ for $u,v
        \in G(\cZ)$. The support of $u$ is contained in $\cZ$ and so
        we can apply the previous lemma to find a $\bar{u}$ such that
        $\bar{u} \in W^{2}$, $\bar{u}$ is supported on $\cB$, and
        $\bar{u}\vert_{\cZ} = u$. Similarly, the support of $v$ is
        contained in $\cB$ and so we can find a $\bar{v}$ such that
        $\bar{v} \in W^{2}$, $\bar{v}$ is supported on $\cA$, and
        $\bar{v}\vert_{\cB}= v$. Thus we have $h = [\bar{u},\bar{v}]
        \in W^{8}$.
\end{proof}

\subsection{Step 4: Pigeonhole}\label{ssec:step4S}

Next we see that, except for a finite error, we can push any brick into a specific subbrick.

\begin{LEM} \label{lem:pigeonholeS}
	For any brick $\cA$ and subbrick $\cB$ there exists $a_{1},a_{2} \in W^{2}$ so that for all but finitely many annuli $A$ in $\cA$, either $a_{1}A \in \cB$ or $a_{2}A \in \cB$. 
\end{LEM}

\begin{proof}
	Write $\cB = \bigsqcup_{i=1}^{\infty} B_{i}$, with each
        $B_{i}$ a big annulus. Let $\Lambda_\alpha$, $\alpha\in\R$ be
        an uncontable collection of infinite subsets of $\mathbb
        N\smallsetminus\{1\}$ 
        so
        that if $\alpha\neq\beta$, then
        $\Lambda_\alpha\cap\Lambda_\beta$ is finite. Write
        $\cB_{\alpha}$ for the subbrick of $\cB$ consisting of annuli
        indexed by $\Lambda_{\alpha}$. Then for every $\alpha$ choose
        a homeomorphism $f_\alpha\in G(\S_U)$ with the following property. Say $\Lambda_\alpha=\{n_1,n_2,n_3,\cdots\}$ with $1<n_1<n_2<n_3<\cdots$. Then $f_\alpha$ sends $B_{n_i}$ homeomorphically onto the big annulus cobounded by $B_{n_i}$ and $B_{n_{i+1}}$ and this component onto $B_{n_{i+1}}$. Note that these maps exist by \Cref{lem:bigannulusproperties}.
	
	Since the collection $\{\Lambda_{\alpha}\}_{\alpha \in \R}$ is uncountable there exist some pair $\alpha \neq \beta$ so that $f_{\alpha}$ and $f_{\beta}$ are in the same left translate of $W$. See \Cref{fig:pigeonholes} for a schematic of the arrangement of these annuli and the maps $f_{\alpha}$ and $f_{\beta}$.  Thus we have that $f_{\alpha}^{-1} f_{\beta}$ and $f_{\beta}^{-1}f_{\alpha}$ are both in $W^{2}$. Set $a_{1} =f_{\beta}^{-1} f_{\alpha}$ and $a_{2} = f_{\alpha}^{-1}f_{\beta}$. Let $\cA_{0} = \cB_{\alpha} \cap \cB_{\beta}$, note that this set is finite, let $\cA_{1} = \cA \setminus \cB_{\alpha}$ and $\cA_{2} = \cA \setminus \cB_{\beta}$. Note that $\cA = \cA_{0} \cup \cA_{1} \cup \cA_{2}$. Now given any annulus $A$ in $\cA$ that is not in the finite exceptional set $\cA_{0}$ we have that either $A \in \cA_{1}$ or $A \in \cA_{2}$. If $A \in \cA_{1}$, then $a_{1}A \in \cB_{\beta} \subset \cB$ and if $A \in \cA_{2}$, then $a_{2}A \in \cB_{\alpha} \subset \cB$. 
	
	\begin{figure}[ht!]
	    \centering
	    \def\svgwidth{\textwidth}

\begingroup%
  \makeatletter%
  \providecommand\color[2][]{%
    \errmessage{(Inkscape) Color is used for the text in Inkscape, but the package 'color.sty' is not loaded}%
    \renewcommand\color[2][]{}%
  }%
  \providecommand\transparent[1]{%
    \errmessage{(Inkscape) Transparency is used (non-zero) for the text in Inkscape, but the package 'transparent.sty' is not loaded}%
    \renewcommand\transparent[1]{}%
  }%
  \providecommand\rotatebox[2]{#2}%
  \newcommand*\fsize{\dimexpr\f@size pt\relax}%
  \newcommand*\lineheight[1]{\fontsize{\fsize}{#1\fsize}\selectfont}%
  \ifx\svgwidth\undefined%
    \setlength{\unitlength}{496.77293197bp}%
    \ifx\svgscale\undefined%
      \relax%
    \else%
      \setlength{\unitlength}{\unitlength * \real{\svgscale}}%
    \fi%
  \else%
    \setlength{\unitlength}{\svgwidth}%
  \fi%
  \global\let\svgwidth\undefined%
  \global\let\svgscale\undefined%
  \makeatother%
  \begin{picture}(1,0.34853168)%
    \lineheight{1}%
    \setlength\tabcolsep{0pt}%
    \put(0,0){\includegraphics[width=\unitlength,page=1]{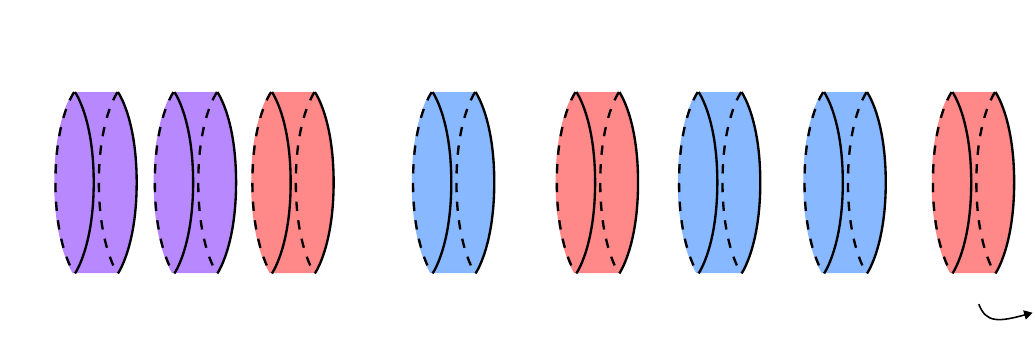}}%
    \put(0.09319133,0.27333329){\color[rgb]{0,0,0}\makebox(0,0)[lt]{\lineheight{0}\smash{\begin{tabular}[t]{l}$\mathcal{B}_{\alpha}\cap\mathcal{B}_{\beta}$\\\end{tabular}}}}%
  \put(0,0){\includegraphics[width=\unitlength,page=2]{pigeonholes.pdf}}%
    \put(0.36929333,0.32622826){\color[rgb]{0,0,0}\makebox(0,0)[lt]{\lineheight{0}\smash{\begin{tabular}[t]{l}$f_{\alpha}$\end{tabular}}}}%
    \put(0.51527832,0.00600423){\color[rgb]{0,0,0}\makebox(0,0)[lt]{\lineheight{0}\smash{\begin{tabular}[t]{l}$f_{\beta}$\end{tabular}}}}%
    \put(0.32205015,0.21801814){\color[rgb]{0,0,0}\makebox(0,0)[lt]{\lineheight{0}\smash{\begin{tabular}[t]{l}$\mathcal{B}_{\alpha}$\\\end{tabular}}}}%
    \put(0.47781879,0.10630789){\color[rgb]{0,0,0}\makebox(0,0)[lt]{\lineheight{0}\smash{\begin{tabular}[t]{l}$\mathcal{B}_{\beta}$\\\end{tabular}}}}%
  \end{picture}%
\endgroup%

		    \caption{A schematic of the pigeonhole argument. The annuli in purple represent the finitely many annuli in $\cB_{\alpha}\cap \cB_{\beta}$. The red and blue annuli are the annuli of $\cB_{\alpha}$ and $\cB_{\beta}$, respectively.}
	    \label{fig:pigeonholes}
	\end{figure}
\end{proof}

Finally we can put all of these pieces together to see that any map supported on a brick is in $W^{72}$. 

\begin{lemma} \label{lem:bricksteinhausS}
	Every $h \in G(\S_U)$ supported on a brick is in $W^{72}$. 
\end{lemma}

\begin{proof}
	First apply \Cref{lem:EMfragment}
	to $h$ to find $h_{1},h_{2} \in G(\S_U)$ with support on bricks $\cA_{1}$ and $\cA_{2}$, respectively. Abusing notation, we replace $h$ with one of the $h_{i}$ and will show that such a map is in $W^{36}$.
	
	Apply \Cref{lem:diagonal2S} to $\cA$ to find $\cZ$ so that any map in $G(\cZ)$ with alternating support is in $W^{8}$. Next apply \Cref{lem:pigeonholeS} to $\cZ$ as a subbrick of $\cA$ to obtain maps $a_{1},a_{2} \in W^{2}$. Write $\cA = F \sqcup X \sqcup Y$ where $F$ is the finite collection of annuli for which \Cref{lem:pigeonholeS} fails, $X$ is the set of annuli for which $a_{1}A \in \cZ$ for all $A \in X$, and $Y$ is the complement of these two sets. We can write $h = h_{F}h_{X}h_{Y}$ for the restrictions of $h$ to each of these sets. Note that $a_{1}h_{X}a_{1}^{-1}$ and $a_{2}h_{Y}a_{2}^{-1}$ are two maps with alternating support in $\cZ$ and therefore $h_{X},h_{Y} \in W^{12}$.
	
	It remains to deal with $h_{F}$. 
Let $g \in G(\S_{U})$
        be a map that sends the set $F$ of annuli into $\cZ$. By assumption,
        $W^{2}$ is dense in $G(\S_{U'})$ and so we may approximate $g$
        by $g'\in W^2$ (which may not belong to $G(\S_U)$ or even be
        identity outside of $\S_U$).
        Thus we see that $g'h_{F}g'^{-1}$ is again a map
        with alternating support in $\cZ$ and hence $h_{F} \in
        W^{12}$. We conclude that $h \in W^{36}$ and thus our original
        map is contained in $W^{72}$.
\end{proof}

\subsection{Step 5: Wrapping Up}\label{ssec:step5surface}

We again first verify a version of the Steinhaus condition for a
stable neighborhood of a telescoping end. 

\begin{THM}\label{thm:globalpointedS}
	Let $\S_{U}$ be a stable neighborhood of a telescoping end
        $x$. Suppose $W$ is a symmetric set in $G=\Homeo(\S)$ such that 
	\begin{itemize}
		\item $\Homeo(\S)$ is covered by countably many sets $g_{i}W$ with $g_{i} \in \Homeo(\S)$, and 
		\item $W^{2}$ is dense in $G(\S_U)$. 
	\end{itemize}
	Then $G(\S_U) \subset W^{288}$.
\end{THM}

\begin{proof}
	Let $h \in G(\S_U)$. We first apply \Cref{lem:fpfragmentS} and \Cref{lem:brickfragmentS} in order to write $h=h_{1}h_{2}h_{3}h_{4}$ where each $h_{i} \in G(\S_U)$ fixes an end $x_{i} \in G(\S_U) \cdot x$ so that $\S_U$ is stable with respect to $x_{i}$ and the support of $h_{i}$ is contained in a brick around $x_{i}$. Next we apply \Cref{lem:bricksteinhausS} to see that each $h_{i}$ is contained in $W^{72}$. We conclude that $h \in W^{288}$. 
\end{proof}

\begin{RMK}\label{rmk:disks}
	Let $D \subset \S$ be a closed disk. By considering a fixed
        point near the boundary as a marked point we can also obtain a version of the previous theorem for $G(D)$. That is, under the same assumptions, we have that $G(D) \subset W^{288}$. 
\end{RMK}

As in the end space case, we technically need a slight extension of
all of the results in \Cref{ssec:step2S,ssec:step3S,ssec:step4S} and
the above theorem. See \Cref{omnibus} and the discussion immediately
preceding it. In other words, we can run the previous proofs
simultaneosly on any finite collection of pairwise disjoint stable
neighborhoods of telescoping ends.

\begin{PROP}\label{multi}
  Suppose $\{U_i\}$ is a finite collection of pairwise disjoint stable
  neighborhoods of telescoping ends $x_i$, and suppose $\S_{U_i}$ is a
  pairwise disjoint collection of corresponding subsurfaces. Let
  $W\subset \Homeo(\S)$ be a symmetric set, so that countably many translates cover
  $\Homeo(\S)$, and so that $W^2$ is dense in each $G(\S_{U_i})$. Then
  any homeomorphism $h$ of $\S$ which is identity outside of
  $\cup_i\S_{U_i}$ with restrictions to each $\S_{U_i}$ belonging to
  $G(\S_{U_i})$ is in $W^{288}$.
\end{PROP}

We are finally ready to prove one implication in Theorem A.

\begin{thm}\label{thm:surfaceACmain}
	Let $\S$ be a stable surface. If every end of $\S$ is telescoping, then AC holds for $\Homeo(\S)$. 
\end{thm}

\begin{proof}
	We will verify the Steinhaus property. Let $W$ be a symmetric
        subset of $\Homeo(\S)$ so that countably many translates cover
        the whole group. We will show that $W^{4896}$ contains a
        neighborhood of the identity. We can apply \Cref{lem:baire} to see
        that there exists some neighborhood of the identity $\cU
        \subset \Homeo(\S)$ so that $W^{2}$ is dense in $\cU$. The
        plan is to cover the end space by surfaces of the form $\S_U$,
        represent given $h\in \Homeo(\S)$ close to the identity as a
        composition of maps supported on $\S_U$'s (or compact
        subsurfaces), and apply Proposition \ref{multi}. The
        difficulty is that homeomorphisms in $\S_U$ require an extra
        big annulus contained in $\S_U$ on which the map is identity,
        and the realizations of these big annuli have to be pairwise
        disjoint, for otherwise we cannot work simultaneously on all of
        them. 

        {\bf Step 1.} Here we construct a collection of pairwise
        disjoint $\S_{U_i}$ that cover the end space. 
	We temporarily work in the Fruedenthal compactification of
        $\S$, which we denote by $\hat{\S}$. Note that $\Homeo(\S) =
        \Homeo(\hat{\S},E(\S),E_{g}(\S))$, 
        the group of homeomorphisms that fix the ends of $\S$,
        setwise. Fix a metric on $\hat{\S}$. There will be
        $\epsilon>0$ so that if a homeomorphism moves points
        $<\epsilon$ it belongs to $\cU$. Using
        \Cref{lem:stablepartition} we can cover $E(\S)$ by a finite
        collection $U_i$ of pairwise disjoint clopen sets so that each
        $U_i$ is a stable neighborhood of one of its points
        $x_i$. Further, we can construct pairwise disjoint subsurfaces
        $\S_{U_i}$ with one boundary component $\gamma_i$ with end
        space $U_i$. We can also choose these sets and subsurfaces so
        that they have diameter $<\epsilon$. We now decompose the
        given homeomorphism $h$, assumed close to the identity, into 3
        homeomorphism, $h=h_1h_2h_3$, so that each is close to the
        identity (i.e. for every neighborhood $\cV$ of the identity,
        there is a neighborhood $\cV'$ so that if $h\in\cV'$ then
        $h_1,h_2,h_3\in \cV$). The simple closed curves $\gamma_i$ are
        moved only slightly by $h$, so we can choose $h_3$ to be
        supported on pairwise disjoint annular neighborhoods of
        $\gamma_i$ and so that $hh_3^{-1}$ fixes each $\gamma_i$
        pointwise. Then $hh_3^{-1}$ is the composition $h_1h_2$, where
        $h_1$ is a homeomorphism of the compact
        surface bounded by the $\gamma_i$ and it is close to the
        identity, and $h_2$ is supported on the union of
        $\S_{U_i}$. We can then use the Edwards-Kirby fragmentation \cite{EK1971}
        and \Cref{rmk:disks} (performed simultaneously) to see that
        both $h_1,h_3\in W^{864}$ (here we need $864=3\times 288$ because the
        fragmentation produces 3 families of pairwise disjoint
        disks). See also \cite{Mann2016,Rosendal2008}
        We are thus reduced to working with $h_2$ and we rename it to
        $h$.

        {\bf Step 2.} Write $\S_{U_i}=\cup_{j=0}^{\infty} Y^i_j$ as in
        \Cref{lem:telnbhdannuli} and let the boundary of the first
        annulus $Y^i_0$ be $\gamma_i\sqcup\delta_i$. Exactly as in
        Step 1 we can write $h=h_1h_2h_3$ where $h_3$ is supported in
        a neighborhood of the $\delta_i$'s, $h_1$ is supported on the
        annuli $Y^i_0$ and $h_2$ is supported on the stable
        neighborhoods $\cup_{j=1}^{\infty} Y^i_j$. In particular,
        $h_3\in  W^{864}$ as before. Note that $h_2|\S_{U_i}$ belongs
        to $G(\S_{U_i})$ since we arranged that it is identity on the
        first annulus $Y^i_0$. Thus $h_2\in W^{288}$ by Proposition
        \ref{multi}. It remains to deal with $h_1$, which we rename as
        $h$.

        {\bf Step 3.} We now assume $h$ is supported on $\cup
        Y^i_0$. First repeat Step 1 for each annulus $Y^i_0$ and write
        $h=h_1h_2h_3$ with $h_1,h_3\in W^{864}$ and with
        $h_2$ supported on a finite collection of pairwise disjoint
        stable neighborhoods $\S_{V^i_j}\subset Y^i_0$, and in fact we
        may assume that $h_2$ is identity in a neighborhood of each
        boundary $\partial \S_{V^i_j}$, which we call $\gamma^i_j$. Again rename
        $h=h_2$. Choose
        pairwise disjoint $Z^i_j\subset U_i$ near $x_i$ (and so away
        from the ends of $Y^i_0$) homeomorphic to $V^i_j$. Let
        $\S_{Z^i_j}$ be the corresponding surfaces and arrange that
        they are pairwise disjoint, disjoint from all the annuli in
        the support of $h$, and $\S_{Z^i_j}\subset (\S_{U_i}\setminus Y^{i}_{0})$. Choose
        pairwise disjoint arcs $\alpha^i_j$ connecting $\gamma^i_j$
        with $\partial \S_{Z^i_j}$, intersecting the two curves only
        at the endpoints,
        and disjoint from all other such curves. Also, ensure that
        $\alpha^i_j\subset\S_{U_i}$. Use $\alpha^i_j$ to boundary
        connect sum the surface $\S_{V^i_j}$ cut off by $\gamma^i_j$ with
        $\S_{Z^i_j}$. See \Cref{fig:acproofnew} for a schematic picture of this arrangement. The new surface is also a stable neighborhood of the
        same end as $\S_{V^i_j}$, but we now added ``luft'', i.e. an annulus where $h$
        is identity, and now we can apply Proposition \ref{multi} to
        conclude that $h\in W^{288}$. Putting everything together we
        get the power $17\times 288=4896$.
          
          	\begin{figure}[ht!]
	    \centering
	    \def\svgwidth{\textwidth}

\begingroup%
  \makeatletter%
  \providecommand\color[2][]{%
    \errmessage{(Inkscape) Color is used for the text in Inkscape, but the package 'color.sty' is not loaded}%
    \renewcommand\color[2][]{}%
  }%
  \providecommand\transparent[1]{%
    \errmessage{(Inkscape) Transparency is used (non-zero) for the text in Inkscape, but the package 'transparent.sty' is not loaded}%
    \renewcommand\transparent[1]{}%
  }%
  \providecommand\rotatebox[2]{#2}%
  \newcommand*\fsize{\dimexpr\f@size pt\relax}%
  \newcommand*\lineheight[1]{\fontsize{\fsize}{#1\fsize}\selectfont}%
  \ifx\svgwidth\undefined%
    \setlength{\unitlength}{391.91338006bp}%
    \ifx\svgscale\undefined%
      \relax%
    \else%
      \setlength{\unitlength}{\unitlength * \real{\svgscale}}%
    \fi%
  \else%
    \setlength{\unitlength}{\svgwidth}%
  \fi%
  \global\let\svgwidth\undefined%
  \global\let\svgscale\undefined%
  \makeatother%
  \begin{picture}(1,0.44369559)%
    \lineheight{1}%
    \setlength\tabcolsep{0pt}%
    \put(0,0){\includegraphics[width=\unitlength,page=1]{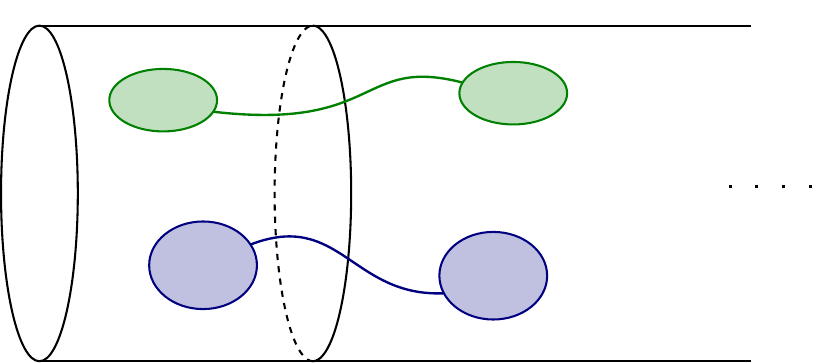}}%
    \put(0.21524009,0.43129417){\color[rgb]{0,0,0}\makebox(0,0)[lt]{\lineheight{1.25}\smash{\begin{tabular}[t]{l}$Y^{i}_{0}$\end{tabular}}}}%
    \put(0.16250917,0.32091703){\color[rgb]{0,0,0}\makebox(0,0)[lt]{\lineheight{1.25}\smash{\begin{tabular}[t]{l}$\Sigma_{V^{i}_{1}}$\end{tabular}}}}%
    \put(0.21109585,0.09757194){\color[rgb]{0,0,0}\makebox(0,0)[lt]{\lineheight{1.25}\smash{\begin{tabular}[t]{l}$\Sigma_{V^{i}_{2}}$\end{tabular}}}}%
    \put(0.60234659,0.3157858){\color[rgb]{0,0,0}\makebox(0,0)[lt]{\lineheight{1.25}\smash{\begin{tabular}[t]{l}$\Sigma_{Z^{i}_{1}}$\end{tabular}}}}%
    \put(0.58444621,0.1163415){\color[rgb]{0,0,0}\makebox(0,0)[lt]{\lineheight{1.25}\smash{\begin{tabular}[t]{l}$\Sigma_{Z^{i}_{2}}$\end{tabular}}}}%
    \put(0.49661557,0.43129417){\color[rgb]{0,0,0}\makebox(0,0)[lt]{\lineheight{1.25}\smash{\begin{tabular}[t]{l}$\Sigma_{U_{i}}\setminus Y^{i}_{0}$\end{tabular}}}}%
    \put(0.43953855,0.30472147){\color[rgb]{0,0,0}\makebox(0,0)[lt]{\lineheight{1.25}\smash{\begin{tabular}[t]{l}$\alpha^{i}_{1}$\end{tabular}}}}%
    \put(0.43953855,0.13125006){\color[rgb]{0,0,0}\makebox(0,0)[lt]{\lineheight{1.25}\smash{\begin{tabular}[t]{l}$\alpha^{i}_{2}$\end{tabular}}}}%
  \end{picture}%
\endgroup%

		    \caption{An example of how to disjointly extend stable neighborhoods in Step 3.}
	    \label{fig:acproofnew}
	\end{figure}
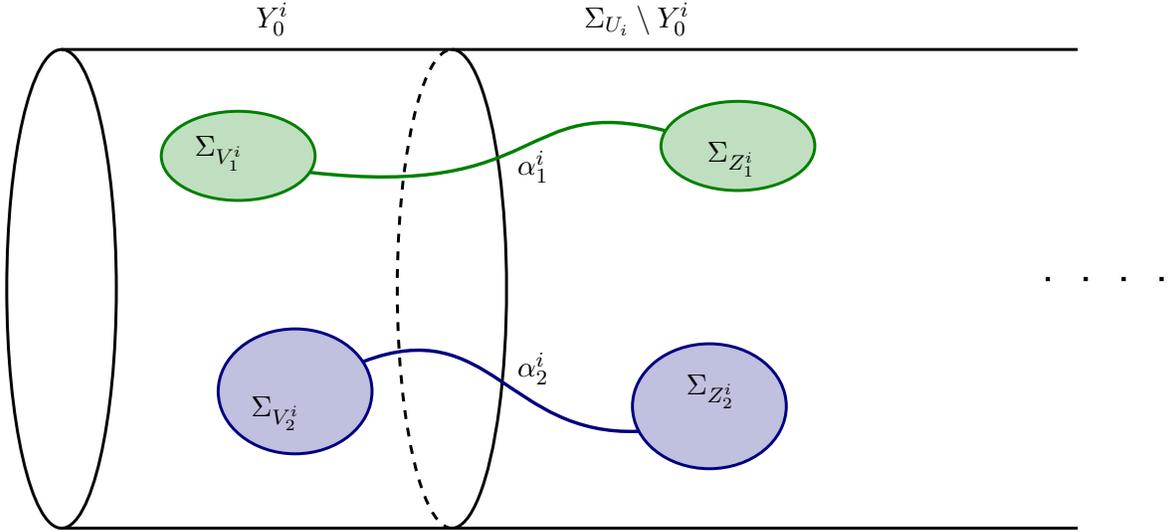
	
 \end{proof}

\subsection{Failure of Automatic Continuity}\label{ssec:acfailure}

In this section we provide tools in order to generate discontinuous homomorphisms. All of our discontinuous homomorphisms will factor through the mapping class group $\Map(\S) = \Homeo(\S)/\Homeo_{0}(\S)$ and appear as an application of the following theorem. We note that this theorem provides a general tool for building discontinuous homomorphisms for closed subgroups of $\Map(\S)$. It is not stated explicitly in this form in \cite{Domat2022}, but follows from the arguments in sections 7,8, and 10. First we need to recall the definition of a nondisplaceable subsurface and a nondisplaceable sequence. 

\begin{DEF}
	Let $K \subset \S$ be a connected subsurface and $G < \Map(\S)$ a subgroup. We say that $K$ is \textbf{$G$-nondisplaceable} if $K$ and  $g(K)$ have an essential intersection for all $g \in G$. A \textbf{$G$-nondisplaceable sequence} is a sequence $\{K_{i}\}_{i=1}^{\infty}$ so that 
	\begin{enumerate}[(i)]
		\item each $K_{i}$ is homeomorphic to a fixed finite-type surface $K$ of sufficient complexity to carry a pseudo-Anosov homeomorphism,
		\item $K_{i} \cap K_{j} = \emptyset$ for all $i \neq j$,
		\item the collection $\{K_{i}\}$ eventually has trivial intersection with every finite-type subsurface of $S$, and
		\item each $K_{i}$ is $G$-nondisplaceable.
	\end{enumerate}
\end{DEF}

\begin{THM}\cite[Sections 7,8, and 10]{Domat2022}\label{thm:discontinuousold}
    If $G$ is a closed subgroup of $\Map(\S)$ and has a $G$-nondisplaceable sequence $\{K_{i}\}_{i=1}^{\infty}$ such that $G$ contains $\Map(K_{i})$ for all $i$, then $G$ fails to have automatic continuity. In particular, there exists a discontinuous homomorphism $G \rightarrow \Q$. 
\end{THM}

\begin{proof}
	Let $\mathcal{K}=\{K_{i}\}_{i=1}^{\infty}$ be a $G$-nondisplaceable sequence so that each $K_{i}$ is homeomorphic to a fixed finite-type surface $K$. Let $f \in \Map(K)$ be a pseudo-Anosov mapping class and write $f_{i} \in G$ for the map defined to be $f$ on $K_{i}$ and the identity elsewhere. Then, for $A=(a_{i})_{i=1}^{\infty}$ any unbounded sequence of natural numbers we form the mapping class
	\begin{align*}
		f_{\mathcal{K},A} \defeq \prod_{i=1}^{\infty} f_{i}^{a_{i}} \in G.
	\end{align*}
	Now we apply \cite[Theorem 7.1]{Domat2022} to see that $f_{\mathcal{K},A}$ cannot be written as a product of commutators in $G$. We remark that \cite[Theorem 7.1]{Domat2022} is stated specifically in the case that $G$ is the closure of compactly supported mapping classes. However, the steps of the proof only require one to have a $G$-nondisplaceable sequence.  In particular, the proof relies on constructing a sequence of Bestvina-Bromberg-Fujiwara projection complexes \cite{BBF2015} and the projection axioms are verified in \cite[Lemma 3.8]{Domat2022} for any collection of pairwise overlapping finite-type subsurfaces in an infinite-type surface. The proof then proceeds by using actions on these projection complexes to define quasimorphisms that coarsely count the exponents $a_{i}$. Once again, this step only relies on the fact that we began with a pseudo-Anosov defined on a finite-type subsurface by appealing to \cite[Proposition 11]{BF2002} and \cite[Proposition 2.9]{BBF2016} (see also \cite[Lemmas 7.4-7.6]{Domat2022}).
	
	Next we run the exact same argument as in \cite[Section
          8]{Domat2022} to see that any such element
        $f_{\mathcal{K},A}$, with $A = (i!)_{i=1}^{\infty}$, generates a copy 
        of $\Q$ in $H_{1}(G;\Z)$. By factoring through the abelianization, we obtain a homomorphism $G \rightarrow \Q$. Finally we note that any finite sub-product of $f_{\mathcal{K},A}$ is trivial in $H_{1}(G;\Z)$ (see \cite[Section 8.1.2]{Domat2022} for technicalities in the low genus case) and hence this map must necessarily be discontinuous. In fact, one can use this to generate $2^{2^{\aleph_{0}}}$ many discontinuous homomorphisms (see \cite[Section 10]{Domat2022}). 
\end{proof}

\subsection{Proof of \Cref{thm:surfaceclassification}}

We need one final lemma that gives an alternate topological characterization of telescoping.

\begin{LEM}\label{lem:telescopefail}
	Let $\S$ be a stable surface. If an end $x \in E(\S)$ is not telescoping, then either 
	\begin{enumerate}[(i)]
		\item $x$ is isolated in $E_{g}(\S)$, 
		\item $x$ has a predecessor $y$ with $E(y)$ countable, or
		\item there exists a family of nested stable neighborhoods $\{U_{n}\}$ descending to $x$ such that each annulus $U_{n}\setminus U_{n+1}$ contains an end $z_{n}$ with $E(z_{n}) \cap (U_{0}\smallsetminus U_{n})  = \emptyset$. 
	\end{enumerate}
\end{LEM}

\begin{proof}
	Let $x \in E(\S)$ be an end that is neither telescoping nor
        isolated in $E_{g}(\S)$. In particular, $x$ is not isolated in
        $E_{g}(\S)$ and $E(x)$ is a discrete set (by
        \Cref{lem:Cantor-type}). Let $U$ be a stable neighborhood of
        $x$ and $U_{k} \searrow x$ a family of nested subsets of $U$
        that descends to $x$. For each $k$, let $M_{k}$ denote the set of equivalence classes in $E(\S)$ that both intersect $U_{k}\setminus U_{k+1}$ and are maximal in $U_{k}\setminus U_{k+1}$. Note that each of $U_{k} \setminus U_{k+1}$ is a stable surface and so $M_{k}$ is finite for each $k$. We now have two cases to consider.
	
	\underline{Case 1:} Suppose $\bigcup_{k=0}^{\infty} M_{k}$ is finite. In other words, there exists some $k_{0}$ such that every equivalence class of end in $U \setminus \{x\}$ intersects $U_{0} \setminus U_{k_{0}}$. This implies that $x$ has finitely many predecessors and hence is a successor. Therefore, since we assumed that $x$ is \emph{not} telescoping, one of these predecessors, $y$, must not be of Cantor type. Thus by \Cref{lem:Cantor-type} we have that $E(y)$ is countable. 

	\underline{Case 2:} Suppose $\bigcup_{k=0}^{\infty} M_{k}$ is infinite. Thus, since each $M_{k}$ is finite, we must have that the annuli $U_{k} \setminus U_{k+1}$ are seeing new equivalence classes of ends as $k\rightarrow \infty$. Therefore, after passing to a subsequence of the $U_{k}$, we can find a family $\{U_{n}\}$ satisfying (iii)
\end{proof}

We are now ready to prove our main classification theorem for surfaces.

\begin{proof}[Proof of \Cref{thm:surfaceclassification}]
	The positive direction is exactly the statement of \Cref{thm:surfaceACmain}. Now we assume that $\S$ has an end $x \in E(\S)$ that is not telescoping. Our discontinuous map will factor through the mapping class group, $\Map(\S)$. 
	
	Since $E(x)$ is not of Cantor type, $E(x)$ is countable by \Cref{lem:Cantor-type} and $x$ is isolated in $E(x)$. Thus we can find a simple separating curve $\gamma$ so that $\gamma$ separates $x$ from all other ends in $E(x)$. Let $\S_{\gamma}$ be the subsurface containing $x$ with boundary $\gamma$. We may assume, without loss of generality, that $E(\S_{\gamma})$ is a stable neighborhood of $x$. Note that the stabilizer, $\Stab(\gamma)$, is a closed countable index subgroup of $\Map(\S)$. Thus by \Cref{lem:induction} it suffices to build a discontinuous homomorphism from $\Stab(\gamma)$ to a countable discrete group. Our new goal is to find a $\Stab(\gamma)$-nondisplaceable sequence in order to apply \Cref{thm:discontinuousold}. The proof breaks down into the three cases from \Cref{lem:telescopefail}
	
	\underline{Case (i):} Suppose that $x\in E_{g}(\S)$ is
        isolated. We claim that any simple curve $\alpha$ in
        $\S_{\gamma}$ that separated $x$ from $\gamma$ is
        $\Stab(\gamma)$-nondisplaceable. Indeed, $\alpha$ cuts
        $\S_{\gamma}$ into two components, one with finite genus and
        one with infinite genus. If $f \in \Stab(\gamma)$ mapped
        $\alpha$ completely into either component, it would change the
        genus of the finite genus piece. As $f$ fixes $\gamma$, this
        cannot happen. Thus we conclude that $f(\alpha) \cap \alpha
        \neq \emptyset$. Finally, to build a
        $\Stab(\gamma)$-nondisplaceable sequence we can take
        sufficiently spaced out sequence of such curves
        $\{\alpha_{i}\}$ that converge to $x$ and for each
        $\alpha_{i}$ take a genus two subsurface $K_{i}$ that contains
        it and is disjoint from all other such subsurfaces.
	
	\underline{Case (ii):} Suppose that $x$ has a countable predecessor, $y$. Write $E'(y) = E(y) \cap \S_{\gamma}$ and note that $E'(y)$ is still countably infinite. Now this case follows exactly as in case 1 except points in $E'(y)$ play the role of genus.
	
	\begin{figure}[ht!]
	    \centering
	    \def\svgwidth{\textwidth}

\begingroup%
  \makeatletter%
  \providecommand\color[2][]{%
    \errmessage{(Inkscape) Color is used for the text in Inkscape, but the package 'color.sty' is not loaded}%
    \renewcommand\color[2][]{}%
  }%
  \providecommand\transparent[1]{%
    \errmessage{(Inkscape) Transparency is used (non-zero) for the text in Inkscape, but the package 'transparent.sty' is not loaded}%
    \renewcommand\transparent[1]{}%
  }%
  \providecommand\rotatebox[2]{#2}%
  \newcommand*\fsize{\dimexpr\f@size pt\relax}%
  \newcommand*\lineheight[1]{\fontsize{\fsize}{#1\fsize}\selectfont}%
  \ifx\svgwidth\undefined%
    \setlength{\unitlength}{409.5841968bp}%
    \ifx\svgscale\undefined%
      \relax%
    \else%
      \setlength{\unitlength}{\unitlength * \real{\svgscale}}%
    \fi%
  \else%
    \setlength{\unitlength}{\svgwidth}%
  \fi%
  \global\let\svgwidth\undefined%
  \global\let\svgscale\undefined%
  \makeatother%
  \begin{picture}(1,0.4574732)%
    \lineheight{1}%
    \setlength\tabcolsep{0pt}%
    \put(0,0){\includegraphics[width=\unitlength,page=1]{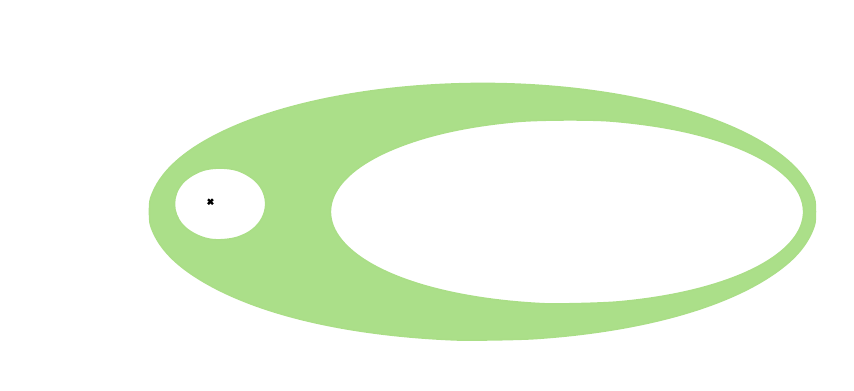}}%
    \put(0.24907461,0.19784686){\color[rgb]{0,0,0}\makebox(0,0)[lt]{\lineheight{0}\smash{\begin{tabular}[t]{l}$z_{n}$\end{tabular}}}}%
    \put(0,0){\includegraphics[width=\unitlength,page=2]{nondisplaceablesucc.pdf}}%
    \put(0.41479749,0.19706542){\color[rgb]{0,0,0}\makebox(0,0)[lt]{\lineheight{0}\smash{\begin{tabular}[t]{l}$z_{n+1}$\end{tabular}}}}%
    \put(0,0){\includegraphics[width=\unitlength,page=3]{nondisplaceablesucc.pdf}}%
    \put(0.07971193,0.19784686){\color[rgb]{0,0,0}\makebox(0,0)[lt]{\lineheight{0}\smash{\begin{tabular}[t]{l}$z_{n-1}$\end{tabular}}}}%
    \put(0,0){\includegraphics[width=\unitlength,page=4]{nondisplaceablesucc.pdf}}%
    \put(0.86497683,0.19706542){\color[rgb]{0,0,0}\makebox(0,0)[lt]{\lineheight{0}\smash{\begin{tabular}[t]{l}$x$\end{tabular}}}}%
    \put(0,0){\includegraphics[width=\unitlength,page=5]{nondisplaceablesucc.pdf}}%
    \put(0.37978134,0.43042203){\color[rgb]{0,0,0}\makebox(0,0)[lt]{\lineheight{0}\smash{\begin{tabular}[t]{l}$\gamma = \partial \Sigma_{\gamma}$\end{tabular}}}}%
    \put(0.48408222,0.37142089){\color[rgb]{0,0,0}\makebox(0,0)[lt]{\lineheight{0}\smash{\begin{tabular}[t]{l}$\partial \Sigma_{U_{n}}$\end{tabular}}}}%
    \put(0.27317091,0.22613208){\color[rgb]{0,0,0}\makebox(0,0)[lt]{\lineheight{0}\smash{\begin{tabular}[t]{l}$\beta_{n}$\end{tabular}}}}%
    \put(0.56608025,0.28060312){\color[rgb]{0,0,0}\makebox(0,0)[lt]{\lineheight{0}\smash{\begin{tabular}[t]{l}$\partial \Sigma_{U_{n+1}}$\end{tabular}}}}%
    \put(0.34975794,0.11007461){\color[rgb]{0,0,0}\makebox(0,0)[lt]{\lineheight{0}\smash{\begin{tabular}[t]{l}$P_{n}$\end{tabular}}}}%
  \end{picture}%
\endgroup%

		    \caption{Building a nondisplaceable pair of pants, $P_{n}$, in case (iii).}
	    \label{fig:nondisplaceablesucc}
	\end{figure}
	
	\underline{Case (iii):} Let $\{U_{n}\}$ be a family of nested
        stable neighborhoods descending to $x$ such that each annulus
        $U_{n} \setminus U_{n+1}$ contains an end $z_{n}$ with
        $E(z_{n}) \cap (U_{0}\smallsetminus U_n) = \emptyset$. We may assume, without loss of generality, that $U_{0} = E(\S_{\gamma})$. We will build a non-displaceable pair of pants for each fixed $n$. Let $P_{n}$ be the pair of pants made up of the curves $\partial \S_{U_{n}}$, $\partial \S_{U_{n+1}}$, and a separating curve, $\beta_{n}$ that cuts off all ends of $U_{n} \setminus U_{n+1}$. See \Cref{fig:nondisplaceablesucc}. We claim that these $P_{n}$ are $\Stab(\gamma)$-nondisplaceable. If $f \in \Stab(\gamma)$, then $f(P_{n})$ cannot land in the component cut off by $\partial \S_{U_{n}}$ since $E(\S_{\gamma}) \setminus U_{n}$ contains no points of $E(z_{n})$. Similarly, if $f(P_{n})$ landed in the component cut off by $\partial \S_{U_{n+1}}$, then $f(E(\S_{\gamma}) \setminus U_{n})$ would intersect $E(z_{n})$. Finally, $f(P_{n})$ cannot land in the component cut off by $\beta_{n}$ since $P_{n}$ must separate $\gamma$ and $x$. As above we can expand these pairs of pants to subsurfaces of sufficiently high complexity and pass to a disjoint subsequence to obtain a $\Stab(\gamma)$-nondisplaceable sequence. 
	
We can now apply \Cref{thm:discontinuousold} to the $\Stab(\gamma)$-nondisplaceable sequences in each of these cases in order to see that $\Stab(\gamma)$ fails to have automatic continuity. Finally, we use \Cref{lem:induction} to conclude that $\Map(\S)$ and hence $\Homeo(\S)$ fails to have automatic continuity. 

\end{proof}

\subsection{Unknown Example} \label{ssec:unknownex}

In this final section we construct an unstable surface $\S$ where our techniques fail to decide automatic continuity. The end space of $\S$ will contain countably many incomparable Cantor sets of ends that themselves converge onto a Cantor set. Building these Cantor sets while avoiding ends that violate \Cref{thm:surfaceclassification} (or the examples of \Cref{rmk:failure}) will take some care.  
         
 We construct the end space of $\Sigma$ inductively. For colored second countable Stone spaces $U_{1},\cdots, U_{k}$, let $C(U_{1},\ldots,U_{k})$ be the Cantor 
set with a copy of $U_{i}$ added into every ``missing interval'' of the Cantor set. Similarly, let $C^{g}(U_{1},\ldots,U_{k})$ be the same except that points in the Cantor set are non-planar.  When $k=0$
we use the notation $C(\emptyset)$ and $C^{g}(\emptyset)$. Let $p$ be an isolated point. We define out 
\emph{level one} cantor sets at follows: 
\[
L_{0,1}= C^g(\emptyset), \quad 
L_{1,1}= C(\emptyset), \quad
L_{2,1}= C^g(p), \quad \text{and}\quad 
L_{3,1} = C(p). 
\]
To reiterate, $L_{0,1}$ is a non-planar Cantor set, $L_{1,1}$ is a planar Cantor set, 
$L_{2,1}$ is a Cantor set with each point accumulated by a sequence of punctures
and $L_{3,1}$ is a non-planar Cantor set where each point accumulated by a sequence of punctures. 
Note that the maximal ends in these space are not comparable. 
Also, every point in these spaces satisfies the assumptions of Theorem~\ref{thm:surfaceclassification}.
	
 Now we can use all non-empty subsets of $\{L_{1,1},L_{2,1},L_{3,1}\}$ to make level two sets. Thus we obtain four level two sets, 
 \[
L_{0,2}= C(L_{1,1},L_{2,1}), \quad 
L_{1,2}= C(L_{2,1},L_{3,1}), \quad
L_{2,2}= C(L_{1,1},L_{3,1}), 
\]
and 
\[
L_{3,2} = C(L_{1,1},L_{2,1}, L_{3,1}). 
\]
Note that $L_{0,1}$ is incomparable to all level two sets. Then we reserve $L_{0,2}$ and again use 
non-empty subsets of $\{L_{1,2},L_{2,2},L_{3,2}\}$ to construct level three sets. We recursively continue 
this process to construct sets at all levels. That is, for all $n\geq 1$, 
 \[
L_{0,n+1}= C(L_{1,n},L_{2,n}), \quad 
L_{1,n+1}= C(L_{2,n},L_{3,n}), \quad
L_{2,n+1}= C(L_{1,n},L_{3,n}), 
\]
and 
\[
L_{3,n+1} = C(L_{1,n},L_{2,n}, L_{3,n}). 
\]
Each time we reserve $L_{0,n}$ so the maximal points in $L_{0,n}$ are not comparable 
to any point in $L_{i,m}$ for $m \geq n$. Thus we obtain the countable collection, $L_{0,1},L_{0,2},L_{0,3},\ldots$, 
of incomparable Cantor sets. 
	
Finally, for each $n \in \N$, we let $\Lambda_{n}$ be the surface with end space $L_{0,n}$ and a single 
boundary component and infinite genus where the ends labeled non-planar are accumulated by genus. 
We then construct $\S$ by beginning with the surface obtained by thickening a rooted binary tree, removing a disk from each pair of pants, and, along the missing disks at level $n$, gluing a copy of $\Lambda_{n}$. Each end of $\S$ is either an isolated puncture or of Cantor type. 
Furthermore, $\S$ has countably many maximal types of ends, corresponding to each of the levels $L_{0,n}$, and one additional maximal Cantor type of end, $L_{\infty}$, which is accumulated by the $L_{0,n}$ as $n \rightarrow \infty$. The ends of $L_{\infty}$ are non-telescoping and so we cannot run our argument for automatic continuity. However, we are also not able to find a nondisplaceable sequence of subsurfaces in order to apply \Cref{thm:discontinuousold} to build a discontinuous map. As such, it is still open whether $\Homeo(\S)$ or $\Map(\S)$ has automatic continuity.


\begin{thebibliography}{10}

\bibitem{AIMPL}
Aim{PL}: Surfaces of infinite type.
\newblock \url{http://aimpl.org/genusinfinity}.

\bibitem{Anderson1958}
R.~D. Anderson.
\newblock The algebraic simplicity of certain groups of homeomorphisms.
\newblock {\em Amer. J. Math.}, 80:955--963, 1958.

\bibitem{Bass1963}
Hyman Bass.
\newblock Big projective modules are free.
\newblock {\em Illinois J. Math.}, 7:24--31, 1963.

\bibitem{BBF2015}
Mladen Bestvina, Ken Bromberg, and Koji Fujiwara.
\newblock Constructing group actions on quasi-trees and applications to mapping
  class groups.
\newblock {\em Publ. Math. Inst. Hautes \'Etudes Sci.}, 122:1--64, 2015.

\bibitem{BBF2016}
Mladen Bestvina, Ken Bromberg, and Koji Fujiwara.
\newblock Stable commutator length on mapping class groups.
\newblock {\em Ann. Inst. Fourier (Grenoble)}, 66(3):871--898, 2016.

\bibitem{BF2002}
Mladen Bestvina and Koji Fujiwara.
\newblock Bounded cohomology of subgroups of mapping class groups.
\newblock {\em Geom. Topol.}, 6:69--89, 2002.

\bibitem{CPV2021}
Yassin Chandran, Priyam Patel, and Nicholas~G. Vlamis.
\newblock Infinite-type surfaces and mapping class groups: Open problems, 2021.
\newblock Available at
  http://qcpages.qc.cuny.edu/~nvlamis/Papers/InfTypeProblems.pdf.

\bibitem{CC2019}
Gregory~R. Conner and Samuel~M. Corson.
\newblock A note on automatic continuity.
\newblock {\em Proc. Amer. Math. Soc.}, 147(3):1255--1268, 2019.

\bibitem{Dickmann2023}
Ryan Dickmann.
\newblock Automatic continuity of pure mapping class groups.
\newblock {\em arXiv preprint arXiv:2306.02599}, 2023.

\bibitem{Domat2022}
George Domat.
\newblock Big pure mapping class groups are never perfect.
\newblock {\em Math. Res. Lett.}, 29(3):691--726, 2022.
\newblock Appendix with Ryan Dickmann.

\bibitem{EK1971}
Robert~D. Edwards and Robion~C. Kirby.
\newblock Deformations of spaces of imbeddings.
\newblock {\em Ann. of Math. (2)}, 93:63--88, 1971.

\bibitem{HHRV}
Jes\'{u}s Hern\'{a}ndez~Hern\'{a}ndez, Michael Hru\u{s}\'{a}k, Christian
  Rosendal, and Ferr\'{a}n Valdez.
\newblock Ample generics in homeomorphism groups of countable ordinals.
\newblock In progress.

\bibitem{Kerekjarto1923}
B{\'e}la Ker{\'e}kj{\'a}rt{\'o}.
\newblock {\em Vorlesungen uber Topologie}.
\newblock J. Springer, 1923.

\bibitem{Mann2016}
Kathryn Mann.
\newblock Automatic continuity for homeomorphism groups and applications.
\newblock {\em Geom. Topol.}, 20(5):3033--3056, 2016.
\newblock With an appendix by Fr\'{e}d\'{e}ric Le Roux and Mann.

\bibitem{Mann2024}
Kathryn Mann.
\newblock {Automatic continuity for homeomorphism Groups and big mapping class
  groups}.
\newblock {\em Michigan Mathematical Journal}, 74(1):215 -- 224, 2024.

\bibitem{MR2023}
Kathryn Mann and Kasra Rafi.
\newblock Large-scale geometry of big mapping class groups.
\newblock {\em Geom. Topol.}, 27(6):2237--2296, 2023.

\bibitem{MR2024}
Kathryn Mann and Kasra Rafi.
\newblock {Two Results on End Spaces of Infinite Type Surfaces}.
\newblock {\em Michigan Mathematical Journal}, 74(5):1109 -- 1116, 2024.

\bibitem{MM1999}
Howard~A. Masur and Yair~N. Minsky.
\newblock Geometry of the complex of curves. {I}. {H}yperbolicity.
\newblock {\em Invent. Math.}, 138(1):103--149, 1999.

\bibitem{Mazur1959}
Barry Mazur.
\newblock On the structure of certain semi-groups of spherical knot classes.
\newblock {\em Inst. Hautes \'Etudes Sci. Publ. Math.}, 1959:111--119, 1959.

\bibitem{MS1920}
Stefan Mazurkiewicz and Wacław Sierpiński.
\newblock Contribution à la topologie des ensembles dénombrables.
\newblock {\em Fundamenta Mathematicae}, 1(1):17--27, 1920.

\bibitem{Richards1963}
Ian Richards.
\newblock On the classification of noncompact surfaces.
\newblock {\em Transactions of the American Mathematical Society},
  106(2):259--269, 1963.

\bibitem{Rosendal2008}
Christian Rosendal.
\newblock Automatic continuity in homeomorphism groups of compact 2-manifolds.
\newblock {\em Israel J. Math.}, 166:349--367, 2008.

\bibitem{Rosendal2009}
Christian Rosendal.
\newblock Automatic continuity of group homomorphisms.
\newblock {\em Bull. Symbolic Logic}, 15(2):184--214, 2009.

\bibitem{RS2007}
Christian Rosendal and S\l{}awomir Solecki.
\newblock Automatic continuity of homomorphisms and fixed points on metric
  compacta.
\newblock {\em Israel J. Math.}, 162:349--371, 2007.

\bibitem{RCS2024}
Christian Rosendal and Luis~Carlos Suarez.
\newblock Aspects of automatic continuity.
\newblock {\em arXiv preprint arXiv:2406.12143}, 2024.

\bibitem{Stone1936}
M.~H. Stone.
\newblock The theory of representations for {B}oolean algebras.
\newblock {\em Trans. Amer. Math. Soc.}, 40(1):37--111, 1936.

\bibitem{Vlamis2024}
Nicholas~G. Vlamis.
\newblock {\em Homeomorphism Groups of Self-Similar 2-Manifolds}, pages
  105--167.
\newblock Springer International Publishing, Cham, 2024.

\end{thebibliography}
\bibliographystyle{plain}

\end{document}